\documentclass[11pt]{article}

\usepackage[hidelinks]{hyperref}
\hypersetup{
  colorlinks   = true, 
  urlcolor     = blue, 
  linkcolor    = blue, 
  citecolor   = red 
}
\usepackage{amsmath,amsthm,amssymb}
\usepackage{graphicx}
\usepackage{bbm}
\newcommand{\sy}{\boldsymbol{\Psi}}
\newcommand{\py}{\boldsymbol{\Phi}}
\newcommand{\T}{\mathbb{T}^3}

\usepackage{xcolor}
\usepackage[margin=1in]{geometry} 
\usepackage{enumerate,mathtools,mathrsfs,bm,graphicx}
\usepackage[numbers, super]{natbib}
\usepackage[hidelinks]{hyperref}
\newcommand{\N}{\mathbb{N}}									

\newcommand{\R}{\mathbb{R}}

\newcommand{\vertiii}[1]{{\left\vert\kern-0.25ex\left\vert\kern-0.25ex\left\vert #1 
    \right\vert\kern-0.25ex\right\vert\kern-0.25ex\right\vert}}
    
\newcommand{\inner}[2]{\left\langle #1, #2 \right\rangle}
\newcommand{\infsum}[1]{\sum_{#1=1}^\infty}					
\DeclarePairedDelimiter\abs{\lvert}{\rvert}					
\newcommand{\norm}[1]{\left\Vert #1 \right\Vert}				

\newcommand{\nl}[2]{\mathcal{L}_{#1}#2}
\newcommand{\Gi}{\mathcal{G}_{i}}

\newcommand{\proj}{\mathcal{P}}
\newcommand{\E}{\mathbbm{E}}

\newtheorem{theorem}{Theorem}[section]

\newtheorem{lemma}[theorem]{Lemma}
\newtheorem{proposition}[theorem]{Proposition}
\newtheorem*{remark}{Remark}
\newtheorem{definition}[theorem]{Definition}

\begin{document}
	\title{Closed Estimates of Leray Projected Transport Noise and Strong Solutions of the Stochastic Euler Equations}
	\author{Daniel Goodair\footnote{\'{E}cole Polytechnique F\'{e}d\'{e}rale de Lausanne, daniel.goodair@epfl.ch}}
	\date{\today} 
	\maketitle
\setcitestyle{numbers}	
\thispagestyle{empty}
\begin{abstract}
We consider the incompressible Euler and Navier-Stokes equations on the three dimensional torus, in velocity form, perturbed by a transport or transport-stretching Stratonovich noise. Closed control of the noise contributions in energy estimates are demonstrated, for any positive integer ordered Sobolev Space and the equivalent Stokes Space; difficulty arises due to the presence of the Leray Projector disrupting cancellation of the top order derivative. This is particularly pertinent in the case of a transport noise without stretching, where the vorticity form cannot be used. As a consequence we obtain, for the first time, the existence of a local strong solution to the corresponding stochastic Euler equation. Furthermore, smooth solutions are shown to exist until blow-up in $L^1\left([0,T];W^{1,\infty}\right)$. 
\end{abstract}
	
\tableofcontents
\textcolor{white}{Hello}
\thispagestyle{empty}

\newpage

\setcounter{page}{1}

\section{Introduction} \label{section introduction}

We are concerned with the incompressible Euler and Navier-Stokes equations on the three dimensional torus, in velocity form, perturbed by a transport or transport-stretching Stratonovich noise.  The stochastic Navier-Stokes equation reads as
\begin{equation} \label{general strati}
    u_t = u_0 - \int_0^t\mathcal{L}_{u_s}u_s\ ds + \nu\int_0^t \Delta u_s\, ds + \int_0^t \mathcal{G}u_s \circ d\mathcal{W}_s - \nabla \rho_t 
\end{equation}
where $\mathcal{L}$ denotes the usual nonlinear term, $\nu > 0$ the viscosity, $\rho$ the pressure, $\mathcal{W}$ is a Cylindrical Brownian Motion and $\mathcal{G}$ represents a transport or transport-stretching noise. The relevant functional analytic and stochastic preliminaries are given in Subsections \ref{subs functional anal} and \ref{subs stoch prelims}. When $\nu$ is formally set to zero we retrieve the stochastic Euler equation. Analysis of the equation (\ref{general strato}) classically hinges upon two manipulations: projection onto the divergence-free and mean-zero subspace by the Leray Projector $\mathcal{P}$, and conversion from Stratonovich to It\^{o} form. To eliminate the semi-martingale pressure before a conversion to It\^{o} form, where the pressure would appear in the cross-variation, we first project the equation to arrive at
\begin{equation} \label{general strato}
    u_t = u_0 - \int_0^t\mathcal{P}\mathcal{L}_{u_s}u_s\ ds - \nu\int_0^t A u_s\, ds + \int_0^t \mathcal{P}\mathcal{G}u_s \circ d\mathcal{W}_s 
\end{equation}
where $A = -\mathcal{P}\Delta$ is the Stokes Operator. The mapping $\mathcal{G}$ is defined along the components of $\mathcal{W}$ through a collection of smooth vector fields $(\xi_i)$, where here the $i^{\textnormal{th}}$ component of $\mathcal{G}$, $\mathcal{G}_i$, represents either the transport operator $\mathcal{L}_{\xi_i}$ defined on a vector field $f$ by 
\begin{equation} \label{def of transport} \mathcal{L}_{\xi_i}f = (\xi_i \cdot \nabla)f = \sum_{j=1}^3\xi_i^j\partial_j f
\end{equation}
or the transport-stretching operator $B_i$,
\begin{equation} \label{def of trans stretch}B_if := \left(\mathcal{L}_{\xi_i} + \mathcal{T}_{\xi_i} \right)f := \sum_{j=1}^3\left(\xi_i^j \partial_j f + f^j\nabla \xi_i^j\right)
\end{equation}
where the superscript $j$ denotes the $j^{\textnormal{th}}$ component of the vector field. The corresponding It\^{o} form of (\ref{general strato}), understood rigorously in [\cite{goodair2024stochastic}] Subsection 3.4, is  
\begin{equation} \label{general itoo}
    u_t = u_0 - \int_0^t\mathcal{P}\mathcal{L}_{u_s}u_s\ ds - \nu\int_0^t A u_s\, ds + \frac{1}{2}\int_0^t \sum_{i=1}^\infty \left(\mathcal{P}\mathcal{G}_i\right)^2u_s \, ds + \int_0^t \mathcal{P}\mathcal{G}u_s d\mathcal{W}_s. 
\end{equation}
Equations of the form (\ref{general itoo}), alongside related models, have been studied in a multitude of texts. Transport noise in fluids has found classical motivations across the works [\cite{brzezniak1992stochastic}, \cite{kraichnan1968small}, \cite{mikulevicius2001equations}, \cite{mikulevicius2004stochastic}], with a surge in popularity over the last ten years due to several intertwining developments. Firstly we mention geometric variational principles as in Holm's [\cite{holm2015variational}] and furthermore [\cite{crisan2022variational}, \cite{holm2021stochastic}, \cite{street2021semi}] whereby intrinsic properties of the fluid are preserved such as Kelvin's Circulation Theorem, as well as a Lagrangian Reynolds Decomposition and Transport Theorem initiated by M\'{e}min [\cite{memin2014fluid}] and expanded upon in [\cite{chapron2018large}, \cite{debussche2024variational}, \cite{resseguier2017geophysical}]. These theoretical derivations are well supported by numerical analysis and data assimilation across the works [\cite{chapron2024stochastic}, \cite{cotter2020data}, \cite{cotter2019numerically},  \cite{crisan2023implementation}, \cite{crisan2023theoretical}, \cite{dufee2022stochastic}, \cite{ephrati2023data}] amongst many others. Another perspective owes to the approach of stochastic model reduction, where noise in the observed large scale structure is introduced through its addition at small scales following some prescribed interaction between them. The notion of scale size is meant in space and time; with this approach rooted in [\cite{majda2001mathematical}] one may have in mind the influence of fast-varying, spatially-localised weather on the overall climate. A Stratonovich transport noise in the Euler and Navier-Stokes equations is derived in this regime, through an infinite scale separation limit, in the papers [\cite{debussche2024second}, \cite{flandoli20212d}, \cite{flandoli2022additive}]. To conclude our motivations let us also mention the potential regularising properties serving as an attraction to transport noise, see for example [\cite{agresti2024delayed}, \cite{coghi2023existence},  \cite{flandoli2021delayed}, \cite{flandoli2010well}, \cite{flandoli2021high}, \cite{galeati2023weak}]. These developments are comprehensively reviewed in [\cite{flandoli2023stochastic}].\\

Despite this attention, the solution theory for such equations still lacks fundamental results from their deterministic counterparts. Where the Stratonovich integral of (\ref{general strati}) is replaced by an It\^{o} integral, one must demand strong conditions on the spatial correlation functions $(\xi_i)$ such as smallness with respect to $\nu > 0$ or parabolicity, see for example [\cite{agresti2024critical}, \cite{agresti2024stochastic}, \cite{brzezniak2021well}, \cite{mikulevicius2004stochastic}]. Henceforth we restrict our discussion to the case of a Stratonovich noise where such assumptions are not imposed. To survey the literature let us contemplate the available \textit{a priori} estimates for the system, which play a fundamental role in the solution theory. Applying the It\^{o} Formula to (\ref{general itoo}), estimates on the $W^{k,2}\left(\mathbb{T}^3;\R^3 \right)$ norm of $u$ rely on taming certain contributions from the noise. We emphasise the terms arising from the It\^{o}-Stratonovich Corrector and quadratic variation of the stochastic integral, yielding
\begin{equation} \label{key noise to control}
    \inner{(\mathcal{P}\mathcal{G}_i)^2u_s}{u_s}_{W^{k,2}} + \norm{\mathcal{P}\mathcal{G}_iu_s}_{W^{k,2}}^2.
\end{equation}
Ultimately we deal with the sum over all $i$ of terms (\ref{key noise to control}), though a decay on $(\xi_i)$ is assumed so that this is unproblematic. In the absence of $\mathcal{P}$ a closed control on (\ref{key noise to control}) is well understood, in the sense that we can bound (\ref{key noise to control}) by $c\norm{u_s}_{W^{k,2}}^2$ where $c$ is dependent on some smooth norm of $\xi_i$; this was shown in [\cite{goodair20233d}] Proposition 2.6 for the transport-stretching noise $B_i$, whilst a similar control on transport and Lie derivative operators has been known since the works [\cite{gyongy1989approximation}, \cite{gyongy2003splitting}, \cite{gyongy1992stochastic}] and has been shown for general pseudo-differential operators in [\cite{tang2023stochastic}]. These results give positive indications towards the solution theory in vorticity form where the Leray Projector is not present, which has led to the existence of local strong solutions of the 3D stochastic Euler equations with a Lie derivative noise in [\cite{crisan2019solution}] and global strong solutions of the 2D stochastic Euler equations with a transport noise in [\cite{lang2023well}]. Introduction of this noise at the level of vorticity follows the principle of Stochastic Advection by Lie Transport developed by Holm in [\cite{holm2015variational}], and although this manifests differently in the 2D and 3D vorticity forms both correspond to the transport-stretching term $B_i$ at the level of velocity. This fact is demonstrated in [\cite{goodair2024thesis}] Subchapter 4.4.4.\\

Our first contribution is a closed control on (\ref{key noise to control}) for arbitrary integer valued $k \geq 0$ in the case $\mathcal{G}_i = B_i$, which we comment on in light of the solution theory for the corresponding vorticity form from [\cite{crisan2019solution}, \cite{lang2023well}]. As a first remark let us mention that working directly with the velocity form is in many cases preferable, which is facilitated by our control. Additionally we point to the fact that higher order smoothness in the 3D case was not shown in [\cite{crisan2019solution}]: that is if one improves the Sobolev regularity of the initial condition then does this smoothness persist on the lifetime of solutions? The answer is affirmative and will follow from our estimates, see Theorem \ref{main existence for euler}. Thirdly we point to a fact appreciated in [\cite{crisan2022solution}], a paper in which the authors consider a rough version of the Euler equation (\ref{general strato}) with $\mathcal{G}_i = B_i$. Their solution theory hinges upon estimates for the vorticity form, stating in Remark 3.6 that ``It is not clear how to obtain a priori estimates [on the velocity] directly due to the projection operators'', which we solve here in the stochastic case; consequently in [\cite{crisan2022solution}], ``altering the structure of the operator appearing in the [noise term] even in a multiplicative way directly impacts the structure of the vorticity equation and prevents us from obtaining a priori
solution estimates'', hence the ability to show estimates at the velocity level affords us the flexibility to consider additional noise perturbations. Moreover we not only show closed estimates in the Sobolev Space but also the equivalent Stokes Space, by which we mean estimates of the term 
\begin{equation} \label{key noise to control Stokes}
    \inner{A^{\frac{k}{2}}(\mathcal{P}\mathcal{G}_i)^2u_s}{A^{\frac{k}{2}}u_s}_{L^2} + \norm{A^{\frac{k}{2}}\mathcal{P}\mathcal{G}_iu_s}_{L^2}^2
\end{equation}
recalling that $A=-\mathcal{P}\Delta$ is the Stokes Operator. We elaborate on these spaces in Subsection \ref{subs functional anal}. Their utility arises in the stochastic Navier-Stokes equations, where it is necessary to extract the gain in regularity from the viscous term in the equation. A closed control of (\ref{key noise to control Stokes}) is entirely non-obvious: in [\cite{goodair2024high}] the bound
\begin{equation} \label{silly not to label}\sum_{i=1}^\infty \left(\inner{A^{\frac{k}{2}}(\mathcal{P}B_i)^2u_s}{A^{\frac{k}{2}}u_s}_{L^2} + \norm{A^{\frac{k}{2}}\mathcal{P}B_iu_s}_{L^2}^2\right) \leq c_{\nu}\norm{A^{\frac{k}{2}}u_s}_{L^2}^2 + \frac{\nu}{2} \norm{A^{\frac{k+1}{2}}u_s}_{L^2}^2
\end{equation}
is shown, and in particular the undesired additional term can be controlled by the viscosity which enables the smoothness of solutions to the stochastic Navier-Stokes equation on its lifetime of existence. Unfortunately as $\nu \rightarrow 0$ then $c_{\nu}$ blows up, so such a bound is unstable in the inviscid limit which is ultimately how we wish to construct solutions of the stochastic Euler equation. The inequality (\ref{silly not to label}) was proven by leveraging a second-order bound on the commutator of $B_i$ and $\Delta$, an approach which is doomed to require the $\varepsilon$ of additional regularity (though one which we note even applies for convex domains with free boundary condition in [\cite{goodair2024high}]).\\

Closing the estimates (\ref{key noise to control}), (\ref{key noise to control Stokes}) for $\mathcal{G}_i = B_i$ depends on the critical observation that $B_i$ preserves gradients and therefore that $\mathcal{P}B_i = \mathcal{P}B_i\mathcal{P}$. As a result, one can rewrite (\ref{key noise to control}) as
\begin{equation} \nonumber
    \inner{\mathcal{P}B_i^2u_s}{u_s}_{W^{k,2}} + \norm{\mathcal{P}B_iu_s}_{W^{k,2}}^2
\end{equation}
which, using that $\mathcal{P}$ commutes with derivatives and is an orthogonal projection, is bounded by $\inner{B_i^2u_s}{u_s}_{W^{k,2}} + \norm{B_iu_s}_{W^{k,2}}^2$ so the control without Leray Projection applies. The case of (\ref{key noise to control Stokes}) is less straightforward, achieved by the appreciation that it admits a bound of the form
\begin{equation} \nonumber
   \sum_{\abs{\alpha}, \abs{\beta} \leq k}\left( \inner{D^\alpha B_i^2f}{D^\beta f} +  \inner{D^\alpha B_if}{D^\beta B_if} + \inner{D^\beta B_i^2f}{D^\alpha f} +  \inner{D^\beta B_if}{D^\alpha B_if}\right)
\end{equation}
so there is a symmetry in the appearance of the derivatives which enables the necessary cancellation. This appears to shed little light on the case of $\mathcal{G}_i = \mathcal{L}_{\xi_i}$, as we do not have that $\mathcal{L}_{\xi_i}$ preserves gradients and so $\mathcal{P}$ remains stuck in the middle of $\mathcal{P}\mathcal{L}_{\xi_i}\mathcal{P}\mathcal{L}_{\xi_i}$ which prevents a cancellation of derivatives. Only when $k=0$ has this been managed, or in other words for estimates in $L^2\left(\mathbb{T}^3;\R^3 \right)$; therefore existence results for analytically weak solutions of the stochastic Navier-Stokes equations as in [\cite{goodair2023zero}], and of analytically \textit{very weak} solutions for the stochastic Euler equations as in [\cite{hofmanova2024global}], are known. For $k \geq 1$ this problem has persisted in the literature for many years, emphasised by the fact that some authors have assumed unexpected commutativity assumptions involving the Leray Projector and noise such as in [\cite{crisan2022spatial}] (3.6), [\cite{tang2022general}] Assumption $D_3$, to enable analysis in this regime. This issue was again highlighted in the recent work [\cite{hu2025local}], whose main result is a control on a term of the form (\ref{key noise to control Stokes}) by $c\norm{u_s}_{W^{k+\frac{1}{2},2}}^2$, although where $\mathcal{P}$ is replaced by the \textit{hydrostatic} Leray Projector appropriate for the fractionally dissipated stochastic Primitive equations. Their method involves commutator estimates of the projector and transport operator, which similarly to the commutator bounds used to produce (\ref{silly not to label}), necessarily relies on an extra degree of smoothness. In the spirit of that work one would expect the usual theory for strong solutions of the stochastic Navier-Stokes equations to follow (although only for $k \geq 3$ where the commutator estimate holds), however we would clearly need to close the estimate in $W^{k,2}(\mathbb{T}^3;\R^3)$ to pass to (local) strong solutions of the stochastic Euler equation which has remained open in the literature.\\

We successfully show closed estimates for (\ref{key noise to control}), (\ref{key noise to control Stokes}) in the case $\mathcal{G}_i = \mathcal{L}_{\xi_i}$. To achieve this we in fact fall back to the case of $B_i$, by adding and subtracting the stretching term $\mathcal{T}_{\xi_i}$ at the right moment. Rewriting $\mathcal{L}_{\xi_i} = B_i - \mathcal{T}_{\xi_i}$ allows us to handle $\mathcal{P}$ associated to $B_i$ as discussed, then the remaining $-\mathcal{T}_{\xi_i}$ is zero-order and is introduced at a time so as to not compromise the estimates. We emphasise again that a control for $B_i$ is also new to this work. \\ 

As a result we obtain a unique local strong solution of the stochastic Euler equation 
\begin{equation} \label{general itoo euler}
    u_t = u_0 - \int_0^t\mathcal{P}\mathcal{L}_{u_s}u_s\ ds + \frac{1}{2}\int_0^t \sum_{i=1}^\infty \left(\mathcal{P}\mathcal{L}_{\xi_i}\right)^2u_s \, ds + \sum_{i=1}^\infty \int_0^t \mathcal{P}\mathcal{L}_{\xi_i}u_s dW^i_s 
\end{equation}
where the stochastic integral is written in component form. The existence of strong solutions of (\ref{general itoo euler}) are new, answering the analytical challenges of noise estimates along with an appreciation that the vorticity form does not yield something more amenable as in the case of transport-stretching noise. The solution is obtained from an inviscid limit of corresponding stochastic Navier-Stokes equations, for which the existence is also new. As a specific motivation for considering the noise without stretching term, we at first mention the classical manuscripts [\cite{kraichnan1968small}, \cite{mikulevicius2004stochastic}], and the recent work [\cite{debussche2024second}] where Stratonovich transport noise at the level of velocity is derived in the 3D Navier-Stokes equations in the spirit of stochastic model reduction. Whilst no stretching noise is present there is an additional \textit{It\^{o}-Stokes drift}, which is of independent concern in the solution theory. Let us also mention [\cite{agresti2024global}] which is largely driven by the relevance of transport noise at the level of velocity as opposed to vorticity, whereby a corresponding result [\cite{flandoli2021high}] existed prior.\\ 



In fact, we show the existence of a unique maximal smooth solution of the equation (\ref{general itoo euler}) and the corresponding transport-stretching case. Smoothness is in the sense that the $W^{k,2}(\T;\R^3)$ regularity of the initial condition persists on the lifetime of solutions, where the maximal time $\Theta$ is characterised by the blow-up
$$\int_0^{\Theta}\norm{u_s}_{W^{1,\infty}}ds = \infty$$
whenever the solution is not global. In the deterministic setting, the sharpest known blow-up criterion is due to Beale, Kato and Majda [\cite{beale1984remarks}] demonstrating that solutions exist until explosion of the vorticity in $L^1\left([0,T];L^\infty\right)$. This result was reproduced for the transport-stretching noise in [\cite{crisan2019solution}], though of course relies heavily on access to the vorticity form which we do not have. Without ability to use the vorticity form, the sharpest criterion that we could expect is that solutions exist until blow-up of the velocity in $L^1\left([0,T];W^{1,\infty}\right)$ which is precisely what we prove. This appears to improve the known results for stochastic Euler equations in velocity form; the closest that we are aware of is explosion in $L^{\infty}\left([0,T];W^{1,\infty}\right)$ which was shown in [\cite{glatt2012local}], see also [\cite{alonso2021local}] for an abstract result. The works [\cite{alonso2021local}, \cite{glatt2011cauchy}] arrive at their criterion by showing the existence of global solutions of the equation truncated by a cut-off in the $W^{1,\infty}$ norm. Our method involves delicate stopping time arguments, needing to overcome issues such as the uniform first-hitting control on the approximating sequence of solutions.\\ 

Additionally we deduce that a local strong solution of the It\^{o} equation is a true solution of the Stratonovich equation, in the sense that it satisfies the identity in Stratonovich form in $L^2(\mathbb{T}^3;\R^3)$. We emphasise that there is a `cost of a derivative' in this conversion due to the differential noise operator, meaning that if the solution has $L^2\left([0,T];W^{k,2} \right)$ regularity then we can only deduce that the Stratonovich form holds in $W^{k-3,2}(\T;\R^3)$ even though the It\^{o} solution is well defined in $W^{k-2,2}(\T;\R^3)$. In the Navier-Stokes equations this is an important distinction, as the energy space for strong solutions gives just $L^2\left([0,T];W^{2,2} \right)$ regularity which is only sufficient to justify the Stratonovich identity in $W^{-1,2}$. In our case of Euler, solutions have at least $C\left([0,T];W^{3,2} \right)$ regularity hence they genuinely solve the Stratonovich identity in $L^2(\T;\R^3)$. This seems to be the first work showing the existence of strong solutions to the genuine Stratonovich form of stochastic Euler equations with transport noise, based upon the recent rigorous conversion result of [\cite{goodair2024stochastic}] Theorem 3.4, instead of simply solving the It\^{o} form obtained heuristically from the Stratonovich one. A `very weak' solution has been shown in [\cite{galeati2023weak}], though in passing to strong solutions one meets the issue of identifying
$$\int_0^t \inner{\mathcal{G}u_s}{\phi}\circ d\mathcal{W}_s = \inner{\int_0^t \mathcal{G}u_s\circ d\mathcal{W}_s}{\phi}$$
 whereby the cost of a derivative becomes key once more: see [\cite{goodair2024stochastic}] Theorem 3.7 and the discussion thereafter for further detail. We now overview the structure of the paper.
\begin{itemize}
   \item Section \ref{section introduction}, this section, concludes with various preliminaries and a statement of the main results. The functional analytic preliminaries are given in Subsection \ref{subs functional anal} whilst the stochastic ones are given in \ref{subs stoch prelims}. Some key properties of transport and transport-stretching noise are stated in Subsection \ref{subsection transport and stretching}. The main results are written in Subsection \ref{subs main results}.

    \item Section \ref{section estimates on the noise} concerns the estimates on the noise terms.

    \item Section \ref{section strong solutions euler} addresses the existence, uniqueness and blow-up of solutions to the stochastic Euler equations.

    \item The paper is supplemented by an appendix, Section \ref{appendix}, containing key results from the literature. 
\end{itemize}


\subsection{Functional Analytic Preliminaries} \label{subs functional anal}

We now recap the classical functional framework for the study of the deterministic Navier-Stokes and Euler Equations on the three dimensional Torus $\mathbbm{T}^3$. The nonlinear operator $\mathcal{L}$ is defined for sufficiently regular functions $f,g:\T \rightarrow \R^3$ by $\mathcal{L}_fg:= \sum_{j=1}^3f^j\partial_jg.$ Here and throughout the text we make no notational distinction between differential operators acting on a vector valued function or a scalar valued one; that is, we understand $\partial_jg$ by its component mappings $(\partial_jg)^l := \partial_jg^l$. For any $m \geq 1$, the mapping $\mathcal{L}: W^{m+1,2} \rightarrow W^{m,2}$ defined by $f \mapsto \mathcal{L}_ff$ is continuous. Some more technical properties of the operator are given at the end of this subsection. For the divergence-free condition we mean a function $f$ such that the property $$\textnormal{div}f := \sum_{j=1}^3 \partial_jf^j = 0$$ holds in the sense of weak derivatives. By the zero-average condition we mean a function $f$ such that $\int_{\T}f(x)dx = 0$. We introduce some new notations to incorporate the divergence-free and zero-average restrictions. These facts are all presented in [\cite{robinson2016three}] Sections 1 and 2. Recall that any function $f \in L^2(\mathbb{T}^3;\R^3)$ admits the representation \begin{align} \label{fourier rep}f(x) = \sum_{k \in \mathbb{Z}^3}f_ke^{ik\cdot x}\end{align} whereby each $f_k \in \mathbb{C}^3$ is such that $f_k = \overline{f_{-k}}$. Let us further introduce $L^2_{\sigma}$ the subset of $L^2(\T;\R^3)$ of zero-average functions $f$ whereby for all $k \in \mathbbm{Z}^3$, $k \cdot f_k = 0$ with $f_k$ as in (\ref{fourier rep}). This enforces the divergence-free condition. For general $m \in \N$ we introduce $W^{m,2}_{\sigma}$ as the intersection of $W^{m,2}(\mathbb{T}^3;\R^3)$ respectively with $L^2_{\sigma}$, with $W^{0,2}_{\sigma}$ simply $L^2_{\sigma}$.
We define the Leray Projector $\mathcal{P}$ as the orthogonal projection in $L^2(\mathbb{T}^3;\R^3)$ onto $L^2_{\sigma}$ and the Stokes Operator $A$ by $-\mathcal{P}\Delta$. The Leray Projector commutes with derivatives (\textit{though not true on a bounded domain!}). There exists a collection of functions $(a_k) \in W^{m,2}_{\sigma}$ for all $m \in \N$, which are eigenfunctions of $A$ with eigenvalues $0 < (\lambda_k)$ increasing to infinity, and form an orthonormal basis of $L^2_{\sigma}$.
For every $s \geq 0$, we define $D(A^s)$ as the subspace of functions $f \in L^2_{\sigma}$ such that $$\sum_{k=1}^\infty \lambda_k^{2s}\inner{f}{a_k}^2 < \infty.$$
Furthermore let us define the mapping $A^s: D(A^s) \rightarrow L^2_{\sigma}$ by $$A^s: f \mapsto \sum_{k=1}^\infty \lambda_k^s\inner{f}{a_k}a_k$$ and associated inner product and norm on $D(A^s)$ by $$ \inner{f}{g}_{A^s} = \inner{A^{s}f}{A^{s}g}, \qquad \norm{f}_{A^s}^2 = \inner{f}{f}_{A^s}$$
where $\inner{\cdot}{\cdot}$ denotes $\inner{\cdot}{\cdot}_{L^2(\mathbbm{T}^3;\R^3)}.$ Defining $\mathcal{P}_n$ by $\mathcal{P}_nf = \sum_{k=1}^n\inner{f}{a_k}a_k$, $\mathcal{P}_n$ is an orthogonal projection in $D(A^s)$ and for $f\in D(A^s)$, $(\mathcal{P}_nf)$ converges to $f$ in $D(A^s)$.
Referring again to [\cite{robinson2016three}] Theorem 2.27, we have $D(A^{\frac{m}{2}}) = W^{m,2}_{\sigma}$ and this norm is equivalent to the usual $W^{m,2}(\mathbb{T}^3;\R^3)$ norm. Henceforth we shall refer to the usual $W^{k,p}(\mathbb{T}^3;\R^3)$ spaces by simply $W^{k,p}$.
Thus we may equip $W^{m,2}_{\sigma}$ with $\inner{\cdot}{\cdot}_{A^{\frac{m}{2}}}$ or the usual inner product $\inner{\cdot}{\cdot}_{W^{m,2}}$. 
If $f,g \in D(A^{s})$ then for any $0 \leq s_1, s_2, r_1, r_2$ with $s_1 + s_2 = s$, $ r_1 + r_2 = s$, 
\begin{equation} \label{stokes operator innner property} \inner{A^{s_1}f}{A^{s_2}g} = \inner{A^{r_1}f}{A^{r_2}g}\end{equation}
and $A^{s_2}f \in D(A^{s_1})$, $A^sf = A^{s_1}A^{s_2}f$. Moreover $A^m$ can be simply defined on $W^{2m,2}\left(\mathbb{T}^3;\R^3\right)$ by $(-\mathcal{P}\Delta)^m$ which agrees with the above definition on $W^{2m,2}_{\sigma}$, and can be further expressed as $$A^m f = (-1)^m\mathcal{P}\Delta^mf = (-1)^m\mathcal{P}\sum_{j_1=1}^3 \dots \sum_{j_m=1}^3\partial_{j_1}^2 \dots \partial_{j_m}^2f.$$
Additionally we have that
$$\inner{f}{g}_{A^{\frac{1}{2}}} = \sum_{j=1}^3\inner{\partial_jf}{\partial_jg}$$
so that for $m$ odd, 
\begin{align}
   \nonumber \inner{f}{g}_{A^{\frac{m}{2}}} &= \inner{A^{\frac{1}{2}}A^{\frac{m-1}{2}}f}{A^{\frac{1}{2}}A^{\frac{m-1}{2}}g}    = \sum_{j=1}^3\inner{\partial_jA^{\frac{m-1}{2}}f}{\partial_jA^{\frac{m-1}{2}}g}\\
     &= \sum_{j=1}^3\sum_{k_1=1}^3 \dots \sum_{k_{\frac{m-1}{2}}=1}^3\sum_{l_1=1}^3 \dots \sum_{l_{\frac{m-1}{2}}=1}^3\inner{\partial_j\partial_{k_1}^2 \dots \partial_{k_{\frac{m-1}{2}}}^2f}{\partial_j\partial_{l_1}^2 \dots \partial_{l_{\frac{m-1}{2}}}^2g} \label{odd power inner product}
\end{align}
having happily commuted $\mathcal{P}$ with derivatives and used that $\mathcal{P}f = f$. In the simpler case of $m$ even,
\begin{equation} \label{even power inner product}
     \inner{f}{g}_{A^{\frac{m}{2}}} = \sum_{k_1=1}^3 \dots \sum_{k_{\frac{m}{2}}=1}^3\sum_{l_1=1}^3 \dots \sum_{l_{\frac{m}{2}}=1}^3\inner{\partial_{k_1}^2 \dots \partial_{k_{\frac{m}{2}}}^2f}{\partial_{l_1}^2 \dots \partial_{l_{\frac{m}{2}}}^2g}.
\end{equation}
Returning to the nonlinear term, we recall the critical duality that for $\phi \in W^{1,2}_{\sigma}$, $f,g \in W^{1,2}$,
\begin{equation}\label{wloglhs}\inner{\mathcal{L}_{\phi}f}{g}= -\inner{f}{\mathcal{L}_{\phi}g}\end{equation}
   which further implies that \begin{equation} \label{cancellationproperty'} \inner{\mathcal{L}_{\phi}f}{f}= 0.\end{equation}
To estimate the nonlinear term in higher norms we recall [\cite{glatt2012local}] Lemma  2.1, asserting that for any $m \geq 3$ there exists a constant $C$ such that for all $f \in W^{m,2}_{\sigma}$, $g \in W^{m+1,2}_{\sigma}$,
\begin{align} \label{nonlinear term control}
    \left\vert \sum_{\abs{\alpha} \leq m}\inner{D^{\alpha}\mathcal{P}\mathcal{L}_fg}{D^{\alpha}g}\right\vert \leq C\left(\norm{f}_{W^{1,\infty}}\norm{g}_{W^{m,2}} + \norm{f}_{W^{m,2}}\norm{g}_{W^{1,\infty}} \right)\norm{g}_{W^{m,2}}.
\end{align}
We note that the result continues to hold for $m < 3$, although it is only meaningful provided $f$ and $g$ have sufficient regularity for the right hand side to be finite. Here and throughout the text, for $\alpha = (\alpha_1, \alpha_2, \alpha_3) \in \left(\N \cup \{0\}\right)^3$, $D^{\alpha}$ represents the weak partial differential operator $D^{\alpha} = \partial_{1}^{\alpha_1}\partial_2^{\alpha_2}\partial_3^{\alpha_3}$. We will consider a partial ordering on the three dimensional multi-indices by $\alpha \leq \beta$ if and only if for all $l =1, 2, 3$ we have that $\alpha_l \leq \beta_l$. We extend this to notation $<$ by 
$\alpha < \beta$ if and only if $\alpha \leq \beta$ and for some $l = 1, 2, 3$, $\alpha_l < \beta_l$. In general throughout the manuscript we shall use $c$ as a generic constant changing from line to line, where dependence on any relevant quantities will be made explicit; in place of this we may also use the notation $ f \lesssim g$ to denote $f \leq cg$ for some such $c$.

\subsection{Stochastic Preliminaries} \label{subs stoch prelims}

Let $(\Omega,\mathcal{F},(\mathcal{F}_t), \mathbbm{P})$ be a fixed filtered probability space satisfying the usual conditions of completeness and right continuity. We take $\mathcal{W}$ to be a cylindrical Brownian motion over some Hilbert Space $\mathfrak{U}$ with orthonormal basis $(e_i)$. Recall (e.g. [\cite{lototsky2017stochastic}], Definition 3.2.36) that $\mathcal{W}$ admits the representation $\mathcal{W}_t = \sum_{i=1}^\infty e_iW^i_t$ as a limit in $L^2(\Omega;\mathfrak{U}')$ whereby the $(W^i)$ are a collection of i.i.d. standard real valued Brownian Motions and $\mathfrak{U}'$ is an enlargement of the Hilbert Space $\mathfrak{U}$ such that the embedding $J: \mathfrak{U} \rightarrow \mathfrak{U}'$ is Hilbert-Schmidt and $\mathcal{W}$ is a $JJ^*-$cylindrical Brownian Motion over $\mathfrak{U}'$. Given a process $F:[0,T] \times \Omega \rightarrow \mathscr{L}^2(\mathfrak{U};\mathscr{H})$ progressively measurable and such that $F \in L^2\left(\Omega \times [0,T];\mathscr{L}^2(\mathfrak{U};\mathscr{H})\right)$, for any $0 \leq t \leq T$ we define the stochastic integral $$\int_0^tF_sd\mathcal{W}_s:=\sum_{i=1}^\infty \int_0^tF_s(e_i)dW^i_s,$$ where the infinite sum is taken in $L^2(\Omega;\mathscr{H})$. We can extend this notion to processes $F$ which are such that $F(\omega) \in L^2\left( [0,T];\mathscr{L}^2(\mathfrak{U};\mathscr{H})\right)$ for $\mathbbm{P}-a.e.$ $\omega$ via the traditional localisation procedure. In this case the stochastic integral is a local martingale in $\mathscr{H}$. We defer to [\cite{goodair2024stochastic}] Chapter 2 for further details on this construction and properties of the stochastic integral. The stochastic integral of (\ref{general itoo}) is then understood by $\mathcal{P}\mathcal{G}u_s(e_i) = \mathcal{P}\mathcal{G}_i(u_s)$, see [\cite{goodair2024stochastic}] Subchapter 3.2. We shall make frequent use of the Burkholder-Davis-Gundy Inequality ([\cite{da2014stochastic}] Theorem 4.36), passage of a bounded linear operator through the stochastic integral ([\cite{goodair2024stochastic}] Proposition 2.16) and the It\^{o} Formula (Proposition \ref{Ito formula}).

\subsection{Transport and Transport-Stretching Noise} \label{subsection transport and stretching}

We collect some fundamental properties of the transport and transport-stretching noise $\mathcal{L}_{\xi_i}$ and $B_i$, as defined in (\ref{def of transport}), (\ref{def of trans stretch}). We always assume at least that $\xi_i \in L^2_{\sigma} \cap W^{1,\infty}$ with $\sum_{i=1}^\infty \norm{\xi_i}_{W^{1,\infty}}^2 < \infty$, where additional spatial regularity will be imposed in the statement of the results. These properties are taken from  [\cite{goodair20233d}] Subsection 2.3, where a more complete description is deferred to. Firstly for $k = 0, 1, 2, \dots$, there exists a constant $c$ such that  
\begin{align}
    \label{T_ibound}\norm{\mathcal{T}_{\xi_i}f}_{W^{k,2}}^2 &\leq  c \norm{\xi_i}^2_{W^{k+1,\infty}}\norm{f}^2_{W^{k,2}}\\
    \label{L_ibound} \norm{\mathcal{L}_{\xi_i}f}_{W^{k,2}}^2 &\leq c\norm{\xi_i}^2_{W^{k,\infty}}\norm{f}^2_{W^{k+1,2}}\\
    \label{boundsonB_i} \norm{B_if}_{W^{k,2}}^2 &\leq c\norm{\xi_i}^2_{W^{k+1,\infty}}\norm{f}^2_{W^{k+1,2}}
\end{align}
for $f, \xi_i$ as required by the right hand side. Moreover $\mathcal{T}_{\xi_i}$ is a bounded linear operator on $L^2$ so has adjoint $\mathcal{T}_{\xi_i}^*$ satisfying the same boundedness. $\mathcal{L}_{\xi_i}$ is a densely defined operator in $L^2$ with domain of definition $W^{1,2}$, and has adjoint $\mathcal{L}_{\xi_i}^*$ in this space given by $-\mathcal{L}_{\xi_i}$ with same dense domain of definition. Likewise then $B_i^*$ is the densely defined adjoint $-\mathcal{L}_{\xi_i} + \mathcal{T}_{\xi_i}^*$. We also note from [\cite{goodair20233d}] Lemma 2.7 that $\mathcal{P}B_i = \mathcal{P}B_i\mathcal{P}$ hence $\mathcal{P}B_i^2 = (\mathcal{P}B_i)^2$.

\subsection{Main Results} \label{subs main results}

We split our main results into two, the first being the estimates on the noise and the second being the resultant solution theory for the stochastic Euler equation. Our noise estimates are as follows.

\begin{proposition} \label{main prop}
    Fix $m \in \N$. There exists a constant $c$ such that for all $\xi_i \in W^{m+2,\infty} \cap L^2_{\sigma}$ and $f \in W^{m+2,2}_{\sigma}$,
    \begin{align*}
   \inner{\mathcal{P}B_i^2f}{f}_{A^{\frac{m}{2}}} +  \norm{\mathcal{P}B_if}_{A^{\frac{m}{2}}}^2  &\leq c\norm{\xi_i}_{W^{m +2,\infty}}^2\norm{f}_{A^{\frac{m}{2}}}^2  ,\\
    \inner{\mathcal{P}B_if}{f}_{A^{\frac{m}{2}}}^2 &\leq c\norm{\xi_i}^2_{W^{m +1,\infty}}\norm{f}^4_{A^{\frac{m}{2}}}, \\
    \inner{\mathcal{P}B_i^2f}{f}_{W^{m,2}} +  \norm{\mathcal{P}B_if}_{W^{m,2}}^2  &\leq c\norm{\xi_i}_{W^{m +2,\infty}}^2\norm{f}_{W^{m,2}}^2 ,\\
    \inner{\mathcal{P}B_if}{f}_{W^{m,2}}^2 &\leq c\norm{\xi_i}^2_{W^{m +1,\infty}}\norm{f}^4_{W^{m,2}},\\
    \inner{\left(\mathcal{P}\mathcal{L}_{\xi_i}\right)^2f}{f}_{A^{\frac{m}{2}}} +  \norm{\mathcal{P}\mathcal{L}_{\xi_i}f}_{A^{\frac{m}{2}}}^2  &\leq c\norm{\xi_i}_{W^{m +2,\infty}}^2\norm{f}_{A^{\frac{m}{2}}}^2 ,\\
    \inner{\mathcal{P}\mathcal{L}_{\xi_i}f}{f}_{A^{\frac{m}{2}}}^2 &\leq c\norm{\xi_i}^2_{W^{m,\infty}}\norm{f}^4_{A^{\frac{m}{2}}},\\
       \inner{\left(\mathcal{P}\mathcal{L}_{\xi_i}\right)^2f}{f}_{W^{m,2}} +  \norm{\mathcal{P}\mathcal{L}_{\xi_i}f}_{W^{m,2}}^2  &\leq c\norm{\xi_i}_{W^{m +2,\infty}}^2\norm{f}_{W^{m,2}}^2 ,\\
    \inner{\mathcal{P}\mathcal{L}_{\xi_i}f}{f}_{W^{m,2}}^2 &\leq c\norm{\xi_i}^2_{W^{m,\infty}}\norm{f}^4_{W^{m,2}}.
\end{align*}
\end{proposition}
We stress again that $\mathcal{P}B_i^2 = \left(\mathcal{P}B_i\right)^2$. Proposition \ref{main prop} will be proven in Section \ref{section estimates on the noise}. We state the solution theory in terms of the general stochastic Euler equation
\begin{equation} \label{general itoo euler proper}
    u_t = u_0 - \int_0^t\mathcal{P}\mathcal{L}_{u_s}u_s\ ds + \frac{1}{2}\int_0^t \sum_{i=1}^\infty \left(\mathcal{P}\mathcal{G}_i\right)^2u_s \, ds + \int_0^t \mathcal{P}\mathcal{G}u_s d\mathcal{W}_s
\end{equation}
where $\mathcal{G}_i$ can be taken as either the transport-stretching $B_i$ or the transport $\mathcal{L}_{\xi_i}$, and the stochastic integral is understood in the sense of Subsection \ref{subs stoch prelims}. Henceforth we fix an arbitrary $T>0$, the horizon on which we shall establish our solution theory. 

\begin{definition} \label{v valued local def}
Fix $3 \leq m \in\N$ and let $u_0:\Omega \rightarrow W^{m,2}_{\sigma}$ be $\mathcal{F}_0-$ measurable. A pair $(u,\tau)$ where $\tau$ is a stopping time such that for $\mathbbm{P}-a.e.$ $\omega$, $\tau(\omega) \in (0,T]$, and $u$ is an adapted process in $W^{m,2}_{\sigma}$ such that for $\mathbbm{P}-a.e.$ $\omega$, $u_{\cdot}(\omega) \in C\left([0,T];W^{m,2}_{\sigma}\right)$, is said to be a local $W^{m,2}_{\sigma}-$strong solution of the equation (\ref{general itoo euler proper}) if the identity
\begin{equation} \label{local identity}
      u_t = u_0 - \int_0^{t \wedge \tau}\mathcal{P}\mathcal{L}_{u_s}u_s\ ds + \frac{1}{2}\int_0^{t \wedge \tau} \sum_{i=1}^\infty \left(\mathcal{P}\mathcal{G}_i\right)^2u_s \, ds + \int_0^{t \wedge \tau} \mathcal{P}\mathcal{G}u_s d\mathcal{W}_s. 
\end{equation}
holds $\mathbbm{P}-a.s.$ in $L^2_{\sigma}$ for all $t \in [0,T]$.
\end{definition}

\begin{definition} \label{V valued maximal definition}
A pair $(u,\Theta)$ such that there exists a sequence of stopping times $(\theta_j)$ which are $\mathbbm{P}-a.s.$ monotone increasing and convergent to $\Theta$, whereby $(u_{\cdot \wedge \theta_j},\theta_j)$ is a local $W^{m,2}_{\sigma}-$strong solution of the equation (\ref{general itoo euler proper}) for each $j$, is said to be a maximal $W^{m,2}_{\sigma}-$strong solution of the equation (\ref{general itoo euler proper}) if for any other pair $(v,\Gamma)$ with this property then $\Theta \leq \Gamma$ $\mathbbm{P}-a.s.$ implies $\Theta = \Gamma$ $\mathbbm{P}-a.s.$.
\end{definition}

\begin{definition} \label{v valued maximal unique}
A maximal $W^{m,2}_{\sigma}-$strong solution $(u,\Theta)$ of the equation (\ref{general itoo euler proper}) is said to be unique if for any other such solution $(v,\Gamma)$, then $\Theta = \Gamma$ $\mathbbm{P}-a.s.$ and \begin{equation} \nonumber\mathbbm{P}\left(\left\{\omega \in \Omega: u_{t}(\omega) =  v_{t}(\omega)  \quad \forall t \in [0, \Theta) \right\} \right) = 1. \end{equation}
\end{definition}

\begin{theorem} \label{main existence for euler}
Fix $3 \leq m \in \N$, let $u_0: \Omega \rightarrow W^{m,2}_{\sigma}$ be $\mathcal{F}_0-$measurable and each $\xi_i \in L^2_{\sigma} \cap  W^{m+6,\infty}$ such that $\sum_{i=1}^\infty \norm{\xi_i}_{W^{m+5,\infty}}^2 < \infty$. There exists a unique maximal $W^{m,2}_{\sigma}-$strong solution $(u,\Theta)$ of the equation (\ref{general itoo euler proper}) with the properties that:
\begin{enumerate}
    \item \label{main one} At $\mathbbm{P}-a.e.$ $\omega$ for which $\Theta(\omega)<T$, we have that \begin{equation} \nonumber \int_{0}^{\Theta(\omega)}\norm{u_s(\omega)}_{W^{1,\infty}}ds  = \infty.\end{equation}



    \item \label{main three} For any stopping time $\tau$ such that $(u_{\cdot \wedge \tau},\tau)$ is a local $W^{m,2}_{\sigma}-$strong solution of the equation (\ref{general itoo euler proper}), the identity
    \begin{equation} \nonumber
      u_t = u_0 - \int_0^{t \wedge \tau}\mathcal{P}\mathcal{L}_{u_s}u_s\ ds  + \int_0^{t \wedge \tau} \mathcal{P}\mathcal{G}u_s \circ d\mathcal{W}_s 
\end{equation}
holds $\mathbbm{P}-a.s.$ in $L^2_{\sigma}$ for all $t \in [0,T]$.
\end{enumerate}
\end{theorem}

Theorem \ref{main existence for euler} will be proven in Section \ref{section strong solutions euler}.



\section{Estimates on the Noise} \label{section estimates on the noise}

This section is dedicated to the proof of Proposition \ref{main prop}. The case of transport-stretching noise is dealt with first in Subsection \ref{subs transport-stretching estimates}, followed by transport noise in Subsection \ref{subs transport noise estimates}.

\subsection{Transport-Stretching Noise} \label{subs transport-stretching estimates}

We begin with the estimates in Stokes spaces.

\begin{proposition} \label{transport and stretching}
    Fix $m \in \N$. There exists a constant $c$ such that for all $\xi_i \in W^{m+2,\infty} \cap L^2_{\sigma}$ and $f \in W^{m+2,2}_{\sigma}$,
    \begin{align}
   \inner{\mathcal{P}B_i^2f}{f}_{A^{\frac{m}{2}}} +  \norm{\mathcal{P}B_if}_{A^{\frac{m}{2}}}^2  &\leq c\norm{\xi_i}_{W^{m +2,\infty}}^2\norm{f}_{A^{\frac{m}{2}}}^2 \label{bigbound1} ,\\
    \inner{\mathcal{P}B_if}{f}_{A^{\frac{m}{2}}}^2 &\leq c\norm{\xi_i}^2_{W^{m +1,\infty}}\norm{f}^4_{A^{\frac{m}{2}}}. \label{bigbound2}
\end{align}
\end{proposition}

The proof of Proposition \ref{transport and stretching} will rely heavily on the following lemma.

\begin{lemma} \label{lemma for conservation B_i}
    Let $\alpha, \beta$ be multi-indices and define $\theta = \abs{\alpha} \vee \abs{\beta}$. There exists a constant $c$ such that, for each $\xi_i \in W^{\theta + 2, \infty} \cap L^2_{\sigma}$ and for all $f \in W^{\theta +2,2}$, we have the bounds
    \begin{align}
   \inner{D^\alpha B_i^2f}{D^\beta f} +  \inner{D^\alpha B_if}{D^\beta B_if} + \inner{D^\beta B_i^2f}{D^\alpha f} +  \inner{D^\beta B_if}{D^\alpha B_if}  &\leq c\norm{\xi_i}_{W^{\theta +2,\infty}}^2\norm{f}_{W^{\theta,2}}^2 \label{otherbound1} ,\\
    \left(\inner{D^{\alpha}B_if}{D^{\beta} f} + \inner{D^{\beta}B_if}{D^{\alpha} f}\right)^2 &\leq c\norm{\xi_i}^2_{W^{\theta +1,\infty}}\norm{f}^4_{W^{\theta ,2}}. \label{otherbound2}
\end{align}
\end{lemma}

\begin{proof}
Let us show the first inequality, for which we start by considering the initial two terms. Observe that 
\begin{align}\label{thenevidently} D^\alpha B_{\xi_i}f = \sum_{\alpha' \leq \alpha}B_{D^{\alpha-\alpha'} \xi_i}D^{\alpha'}f
= \sum_{\alpha' < \alpha}B_{D^{\alpha-\alpha'} \xi_i}D^{\alpha'}f + B_{\xi_i}D^{\alpha}f\end{align}
hence
    \begin{align*}D^\alpha B^2_{\xi_i}f = D^\alpha B_{\xi_i}\big(B_{\xi_i}f\big) = \sum_{\alpha' < \alpha}B_{D^{\alpha-\alpha'} \xi_i}D^{\alpha'}B_{\xi_i}f + B_{ \xi_i}D^{\alpha}B_{\xi_i}f
\end{align*}
 which we apply to those initial two terms, reducing them to
$$\Big\langle \sum_{\alpha' < \alpha}B_{D^{\alpha-\alpha'} \xi_i}D^{\alpha'}B_{\xi_i}f + B_{\xi_i}D^\alpha B_{\xi_i}f,D^\beta f\Big\rangle +  \inner{D^\alpha B_{\xi_i}f}{D^\beta B_{\xi_i}f}.$$ 
We further break this up in terms of the adjoint $B_{\xi_i}^*$, \begin{equation}\label{equitino}\Big\langle \sum_{\alpha' < \alpha} B_{D^{\alpha-\alpha'} \xi_i}D^{\alpha'}B_{\xi_i}f,D^\beta f\Big\rangle + \inner{D^\alpha B_{\xi_i}f}{B^*_{\xi_i} D^\beta f} +  \inner{D^\alpha B_{\xi_i}f}{D^\beta B_{\xi_i}f}\end{equation}
where we sum the second and third inner products and using (\ref{thenevidently}), becoming $$\Big\langle \sum_{\alpha' < \alpha} B_{D^{\alpha-\alpha'} \xi_i}D^{\alpha'}B_{\xi_i}f,D^\beta f\Big\rangle + \Big\langle D^\alpha B_{\xi_i}f,B^*_{\xi_i} D^\beta f + \sum_{\beta' < \beta}B_{D^{\beta-\beta'} \xi_i}D^{\beta'}f + B_{\xi_i}D^\beta f\Big\rangle.$$
Simplify by combining $B_{\xi_i}^*$ and $B_{\xi_i}$, noting that $$B_i^*+B_i= \mathcal{L}_{\xi_i}^* + \mathcal{T}_i^* + \mathcal{L}_{\xi_i} + \mathcal{T}_i = \mathcal{T}_i^* + \mathcal{T}_i,$$ we arrive at the expression
$$\Big\langle \sum_{\alpha' < \alpha}B_{D^{\alpha-\alpha'} \xi_i}D^{\alpha'}B_{\xi_i}f,D^\beta f\Big\rangle + \Big\langle D^\alpha B_{\xi_i}f,\big(\mathcal{T}_{\xi_i} + \mathcal{T}_{\xi_i}^*\big) D^\beta f + \sum_{\beta' < \beta}B_{D^{\beta-\beta'} \xi_i}D^{\beta'}f\Big\rangle.$$ 
As we are looking to achieve control with respect to the $W^{\theta,2}$ norm of $f$, then it is the terms with differential operators of order greater than $\theta$ that concern us. Of course this was the motivating factor behind combining $B_{\xi_i}$ and its adjoint, nullifying the additional derivative coming from $\mathcal{L}_{\xi_i}$. There are more higher order terms to go though, and the strategy will be to write these in terms of commutators with a differential operator of controllable order. This will involve considering different aspects of our sum in tandem, which will be helped with (\ref{thenevidently}) reducing our expression again to  \begin{align*}&\Big\langle \sum_{\alpha' < \alpha}B_{D^{\alpha-\alpha'} \xi_i}D^{\alpha'}B_{\xi_i}f,D^\beta f\Big\rangle\\ & \qquad \qquad \qquad \qquad + \Big\langle \sum_{\alpha' < \alpha} B_{D^{\alpha-\alpha'}\xi_i}D^{\alpha'} f + B_{\xi_i}D^{\alpha}f,\big(\mathcal{T}_{\xi_i} + \mathcal{T}_{\xi_i}^*\big) D^\beta f + \sum_{\beta' < \beta}B_{D^{\beta-\beta'} \xi_i}D^{\beta'}f\Big\rangle.\end{align*} Ultimately the terms in the summand are split up, grouping terms that we will cancel and those of lower order, achieving in total that 
\begin{align} \nonumber
\inner{D^\alpha B_i^2f}{D^\beta f} &+  \inner{D^\alpha B_if}{D^\beta B_if}\\ \nonumber
& \qquad = \inner{B_{\xi_i}D^\alpha f}{\big(\mathcal{T}_{\xi_i}+\mathcal{T}^*_{\xi_i}\big) D^\beta f}\\
    \nonumber & \qquad + \Big\langle \sum_{\alpha' < \alpha} B_{D^{\alpha-\alpha'} \xi_i}D^{\alpha'}f ,\big(\mathcal{T}_{\xi_i} + \mathcal{T}_{\xi_i}^*\big) D^\beta f + \sum_{\beta' < \beta} B_{D^{\beta-\beta'} \xi_i}D^{\beta'}f\Big\rangle\\
    \nonumber & \qquad +  \sum_{\alpha' < \alpha}\inner{B_{D^{\alpha-\alpha'} \xi_i}D^{\alpha'}B_{\xi_i}f}{D^\beta f} + \sum_{\beta' < \beta}\inner{B_{\xi_i}D^\alpha f}{B_{D^{\beta-\beta'} \xi_i}D^{\beta'}f}.
\end{align}
From this we deduce an expression for
\begin{equation} \label{the expression}\inner{D^\alpha B_i^2f}{D^\beta f} +  \inner{D^\alpha B_if}{D^\beta B_if} + \inner{D^\beta B_i^2f}{D^\alpha f} +  \inner{D^\beta B_if}{D^\alpha B_if}\end{equation}
by interchanging the roles of $\alpha,\beta$ and summing the result. That is, (\ref{the expression}) is given by the sum of a first expression
\begin{align}
   \label{first term} \inner{B_{\xi_i}D^\alpha f}{\big(\mathcal{T}_{\xi_i}+\mathcal{T}^*_{\xi_i}\big) D^\beta f} + \inner{B_{\xi_i}D^\beta f}{\big(\mathcal{T}_{\xi_i}+\mathcal{T}^*_{\xi_i}\big) D^\alpha f},
\end{align}
a second expression
\begin{align}
  \nonumber  \Big\langle \sum_{\alpha' < \alpha} B_{D^{\alpha-\alpha'} \xi_i}D^{\alpha'}f &,\big(\mathcal{T}_{\xi_i} + \mathcal{T}_{\xi_i}^*\big) D^\beta f + \sum_{\beta' < \beta} B_{D^{\beta-\beta'} \xi_i}D^{\beta'}f\Big\rangle\\ &+ \Big\langle \sum_{\beta' < \beta} B_{D^{\beta-\beta'} \xi_i}D^{\beta'}f ,\big(\mathcal{T}_{\xi_i} + \mathcal{T}_{\xi_i}^*\big) D^\alpha f + \sum_{\alpha' < \alpha} B_{D^{\alpha-\alpha'} \xi_i}D^{\alpha'}f\Big\rangle \label{second term}
\end{align}
and a third expression
\begin{align}
    \nonumber \sum_{\alpha' < \alpha}\inner{B_{D^{\alpha-\alpha'} \xi_i}D^{\alpha'}B_{\xi_i}f}{D^\beta f} &+ \sum_{\beta' < \beta}\inner{B_{\xi_i}D^\alpha f}{B_{D^{\beta-\beta'} \xi_i}D^{\beta'}f} \\&+ \sum_{\beta' < \beta}\inner{B_{D^{\beta-\beta'} \xi_i}D^{\beta'}B_{\xi_i}f}{D^\alpha f} + \sum_{\alpha' < \alpha}\inner{B_{\xi_i}D^\beta f}{B_{D^{\alpha-\alpha'} \xi_i}D^{\alpha'}f} \label{third term}
\end{align}
which we control individually. Firstly for a treatment of (\ref{first term}), in the first term we have that
    \begin{align*} &\inner{B_{\xi_i}D^\alpha f}{\big(\mathcal{T}_{\xi_i}+\mathcal{T}^*_{\xi_i}\big) D^\beta f}\\ & \quad =\inner{(\mathcal{L}_{\xi_i} + \mathcal{T}_{\xi_i})D^{\alpha}f}{(\mathcal{T}_{\xi_i}^* + \mathcal{T}_{\xi_i})D^{\beta}f}\\ & \quad = \inner{\mathcal{L}_{\xi_i}D^{\alpha}f}{\mathcal{T}_{\xi_i}^*D^{\beta}f} + \inner{\mathcal{L}_{\xi_i}D^{\alpha}f}{\mathcal{T}_{\xi_i}D^{\beta}f} + \inner{\mathcal{T}_{\xi_i}D^{\alpha}f}{\mathcal{T}_{\xi_i}^*D^{\beta}f} + \inner{\mathcal{T}_{\xi_i}D^{\alpha}f}{\mathcal{T}_{\xi_i}D^{\beta}f}\\
    & \quad = \Big(\inner{\mathcal{T}_{\xi_i}\mathcal{L}_{\xi_i}D^{\alpha}f}{D^{\beta}f} + \inner{\mathcal{L}_{\xi_i}D^{\alpha}f}{\mathcal{T}_{\xi_i}D^{\beta}f}\Big) + \Big(\inner{\mathcal{T}_{\xi_i}^2D^{\alpha}f}{D^{\beta}f} + \inner{\mathcal{T}_{\xi_i}D^{\alpha}f}{\mathcal{T}_{\xi_i}D^{\beta}f}\Big).
    \end{align*}

Immediately from (\ref{T_ibound}) we obtain that
$$ \abs{ \inner{\mathcal{T}_{\xi_i}^2D^{\alpha}f}{D^{\beta}f} + \inner{\mathcal{T}_{\xi_i}D^{\alpha}f}{\mathcal{T}_{\xi_i}D^{\beta}f}} \leq c\norm{\xi_i}_{W^{1,\infty}}^2\norm{f}_{W^{\theta,2}}^2$$
which we apply twice to verify the control
\begin{align} \nonumber (\ref{first term}) &\leq  \Big(\inner{\mathcal{T}_{\xi_i}\mathcal{L}_{\xi_i}D^{\alpha}f}{D^{\beta}f} + \inner{\mathcal{L}_{\xi_i}D^{\alpha}f}{\mathcal{T}_{\xi_i}D^{\beta}f}\Big)\\ & \qquad + \Big(\inner{\mathcal{T}_{\xi_i}\mathcal{L}_{\xi_i}D^{\alpha}f}{D^{\beta}f} + \inner{\mathcal{L}_{\xi_i}D^{\alpha}f}{\mathcal{T}_{\xi_i}D^{\beta}f}\Big) + c\norm{\xi_i}_{W^{1,\infty}}^2\norm{f}_{W^{\theta,2}}^2. \label{appliedtwice}
\end{align}
Now for the first bracket, we add and subtract a term to have an expression through the commutator of the operators:
    \begin{align}
    \nonumber &\inner{\mathcal{T}_{\xi_i}\mathcal{L}_{\xi_i}D^{\alpha}f}{D^{\beta}f} + \inner{\mathcal{L}_{\xi_i}D^{\alpha}f}{\mathcal{T}_{\xi_i}D^{\beta}f}\\ \nonumber & \qquad = \inner{(\mathcal{T}_{\xi_i}\mathcal{L}_{\xi_i} - \mathcal{L}_{\xi_i}\mathcal{T}_{\xi_i}) D^{\alpha}f}{D^{\beta}f} + \inner{\mathcal{L}_{\xi_i}\mathcal{T}_{\xi_i}D^{\alpha}f}{D^{\beta}f} + \inner{\mathcal{L}_{\xi_i}D^{\alpha}f}{\mathcal{T}_{\xi_i}D^{\beta}f}\\ \label{analignedthing}
    & \qquad = \inner{(\mathcal{T}_{\xi_i}\mathcal{L}_{\xi_i} - \mathcal{L}_{\xi_i}\mathcal{T}_{\xi_i}) D^{\alpha}f}{D^{\beta}f} + \inner{\mathcal{T}_{\xi_i}D^{\alpha}f}{\mathcal{L}_{\xi_i}^*D^{\beta}f} + \inner{\mathcal{L}_{\xi_i}D^{\alpha}f}{\mathcal{T}_{\xi_i}D^{\beta}f}.
    \end{align}
 The commutator is given explicitly on a function $g$ by
    \begin{align*}
        \mathcal{T}_{\xi_i}\mathcal{L}_{\xi_i}g = \mathcal{T}_{\xi_i}\Big(\sum_{j=1}^3\xi_i^j\partial_jg\Big)
        = \sum_{k=1}^3\Big(\sum_{j=1}^3\xi_i^j\partial_jg\Big)^k\nabla \xi_i^k
        = \sum_{k=1}^3\sum_{j=1}^3\xi_i^j\partial_jg^k\nabla \xi_i^k
    \end{align*}
    and
    \begin{align*}
        \mathcal{L}_{\xi_i}\mathcal{T}_{\xi_i}g &= \mathcal{L}_{\xi_i}\Big(\sum_{k=1}^3g^k\nabla \xi_i^k\Big)
        = \sum_{j=1}^3\xi_i^j\partial_j\Big(\sum_{k=1}^3g^k\nabla \xi_i^k\Big)
        = \sum_{j=1}^3\sum_{k=1}^3\xi_i^j\partial_j \big(g^k\nabla \xi_i^k\big)\\
        &= \sum_{j=1}^3\sum_{k=1}^3\Big(\xi_i^j\partial_jg^k \nabla \xi_i^k + \xi_i^jg^k\partial_j\nabla\xi_i^k\Big)
    \end{align*}
    such that $$(\mathcal{T}_{\xi_i}\mathcal{L}_{\xi_i} - \mathcal{L}_{\xi_i}\mathcal{T}_{\xi_i}) g = -\sum_{j=1}^3\sum_{k=1}^3\xi_i^jg^k\partial_j\nabla\xi_i^k$$
which is of zeroth order in $g$ as required. Therefore, we obtain
$$\abs{\inner{(\mathcal{T}_{\xi_i}\mathcal{L}_{\xi_i} - \mathcal{L}_{\xi_i}\mathcal{T}_{\xi_i}) D^{\alpha}f}{D^{\beta}f}} \leq c\norm{\xi_i}_{W^{2,\infty}}^2\norm{f}_{W^{\theta,2}}^2.$$
Returning to (\ref{analignedthing}), by twice applying this bound we have that
\begin{align*}
      &\inner{\mathcal{T}_{\xi_i}\mathcal{L}_{\xi_i}D^{\alpha}f}{D^{\beta}f} + \inner{\mathcal{L}_{\xi_i}D^{\alpha}f}{\mathcal{T}_{\xi_i}D^{\beta}f} + \inner{\mathcal{T}_{\xi_i}\mathcal{L}_{\xi_i}D^{\beta}f}{D^{\alpha}f} + \inner{\mathcal{L}_{\xi_i}D^{\beta}f}{\mathcal{T}_{\xi_i}D^{\alpha}f}\\ 
    & \qquad  \leq c\norm{\xi_i}_{W^{2,\infty}}^2\norm{f}_{W^{\theta,2}}^2 + \inner{\mathcal{T}_{\xi_i}D^{\alpha}f}{\mathcal{L}_{\xi_i}^*D^{\beta}f} + \inner{\mathcal{L}_{\xi_i}D^{\alpha}f}{\mathcal{T}_{\xi_i}D^{\beta}f}\\ & \qquad \qquad \qquad \quad \qquad \qquad + \inner{\mathcal{T}_{\xi_i}D^{\beta}f}{\mathcal{L}_{\xi_i}^*D^{\alpha}f} + \inner{\mathcal{L}_{\xi_i}D^{\beta}f}{\mathcal{T}_{\xi_i}D^{\alpha}f}.
\end{align*}
Using once more that $\mathcal{L}_{\xi_i}^* = -\mathcal{L}_{\xi_i}$, we observe complete cancellation in the four remaining inner products. Altogether, returning to (\ref{appliedtwice}), we have verified that
$$(\ref{first term}) \leq  c\norm{\xi_i}_{W^{2,\infty}}^2\norm{f}_{W^{\theta,2}}^2.$$
We now move on to the next expression, (\ref{second term}), which is only of order $\theta$ in $f$ by design; through Cauchy-Schwarz and the bound (\ref{boundsonB_i}), we happily deduce that
$$(\ref{second term}) \leq c\norm{\xi_i}_{W^{\theta +1,\infty}}^2\norm{f}_{W^{\theta,2}}^2.$$
To bound (\ref{the expression}), it now only remains to show a sufficient control on (\ref{third term}). Combining the sums, we have that
\begin{align*}
    (\ref{third term}) = &\sum_{\alpha' < \alpha}\inner{D^{\beta}f}{B_{D^{\alpha-\alpha'} \xi_i}D^{\alpha'}B_{\xi_i}f + B_{\xi_i}^*B_{D^{\alpha - \alpha'}\xi_i}D^{\alpha'}f}\\ & + \sum_{\beta' < \beta}\inner{D^{\alpha}f}{B_{D^{\beta-\beta'} \xi_i}D^{\beta'}B_{\xi_i}f + B_{\xi_i}^*B_{D^{\beta - \beta'}\xi_i}D^{\beta'}f}.
\end{align*}
We consider a general term in the first summand, \begin{equation} \label{rewriteofprime}\inner{D^\beta f}{B_{D^{\alpha-\alpha'}\xi_i}D^{\alpha'}B_{\xi_i}f + B^*_{\xi_i}B_{D^{\alpha-\alpha'} \xi_i}D^{\alpha'}f}\end{equation} and employing (\ref{thenevidently}) again we see this becomes \begin{align*}& \Big\langle D^\beta f,B_{D^{\alpha-\alpha'}\xi_i}\bigg(\sum_{\gamma < \alpha'}B_{D^{\alpha'-\gamma}\xi_i}D^\gamma f + B_{\xi_i}D^{\alpha'}f\bigg)+ B^*_{\xi_i}B_{D^{\alpha-\alpha'} \xi_i}D^{\alpha'}f\Big\rangle\\
= & \Big\langle D^\beta f,\sum_{\gamma < \alpha'}B_{D^{\alpha-\alpha'}\xi_i}B_{D^{\alpha'-\gamma}\xi_i}D^\gamma f\Big\rangle + \inner{D^\beta f}{B_{D^{\alpha-\alpha'}\xi_i}B_{\xi_i}D^{\alpha'}f + B^*_{\xi_i}B_{D^{\alpha-\alpha'} \xi_i}D^{\alpha'}f}.  \end{align*}
We have split up these terms to make our approach clearer, as the two will be considered separately. Indeed the first term is now of order $\theta$, so the familiar Cauchy-Schwarz and (\ref{boundsonB_i}) establishes that
\begin{equation} \label{standardcontrol} \Big\langle D^\beta f,\sum_{\gamma < \alpha'}B_{D^{\alpha-\alpha'}\xi_i}B_{D^{\alpha'-\gamma}\xi_i}D^\gamma f\Big\rangle \leq c\norm{\xi_i}_{W^{\theta + 1,\infty}}^2\norm{f}_{W^{\theta,2}}^2.\end{equation}
As for the second inner product, we rewrite the right side as $$B_{D^{\alpha-\alpha'} \xi_i}\big((\mathcal{L}_{\xi_i} + \mathcal{T}_{\xi_i})D^{\alpha'}f\big) + (\mathcal{L}^*_{\xi_i} + \mathcal{T}^*_{\xi_i})B_{D^{\alpha-\alpha'} \xi_i}D^{\alpha'}f$$ and further \begin{equation}\label{theonewiththecommutator}\big(B_{D^{\alpha -\alpha'} \xi_i}\mathcal{L}_{\xi_i} - \mathcal{L}_{\xi_i}B_{D^{\alpha-\alpha'} \xi_i}\big)D^{\alpha'}f + B_{D^{\alpha-\alpha'} \xi_i}\mathcal{T}_{\xi_i}D^{\alpha'}f + \mathcal{T}_{\xi_i}^*B_{D^{\alpha-\alpha'} \xi_i}D^{\alpha'}f.\end{equation}
The latter two terms are once more of order $\theta$, leading to a control as in (\ref{standardcontrol}). Now we show explicitly that the commutator from (\ref{theonewiththecommutator}), acting on a function $g$,
\begin{equation} \label{theactualcommutator}
    (B_{D^{\alpha -\alpha'} \xi_i}\mathcal{L}_{\xi_i} - \mathcal{L}_{\xi_i}B_{D^{\alpha-\alpha'} \xi_i})g
\end{equation}
 is of first order, through the expressions
\begin{align*}
    B_{D^{\alpha-\alpha'} \xi_i}\mathcal{L}_{\xi_i}g &= \sum_{j=1}^3\bigg(D^{\alpha-\alpha'}\xi_i^j \partial_j\Big(\sum_{k=1}^3\xi_i^k\partial_k g\Big) + \Big(\sum_{k=1}^3\xi_i^k\partial_k g\Big)^j\nabla D^{\alpha-\alpha'} \xi_i^j\bigg)\\
    &= \sum_{j=1}^3\sum_{k=1}^3\bigg(D^{\alpha-\alpha'}\xi_i^j\partial_j\xi_i^k\partial_kg  + D^{\alpha-\alpha'}\xi_i^j\xi_i^k\partial_j\partial_kg + \xi_i^k\partial_k g^j \nabla D^{\alpha-\alpha'} \xi_i^j\bigg) 
\end{align*}
and
\begin{align*}
    &\mathcal{L}_{\xi_i}B_{D^{\alpha-\alpha'} \xi_i}g\\ &= \sum_{k=1}^3\xi_i^k\partial_k\bigg(\sum_{j=1}^3D^{\alpha-\alpha'} \xi_i^j\partial_jg + g^j \nabla D^{\alpha-\alpha'} \xi_i^j\bigg)\\
    &= \sum_{j=1}^3\sum_{k=1}^3\bigg(\xi_i^k\partial_kD^{\alpha-\alpha'} \xi_i^j \partial_jg + \xi_i^kD^{\alpha-\alpha'}\xi_i^j\partial_k\partial_jg  + \xi_i^k\partial_kg^j \nabla D^{\alpha-\alpha'} \xi_i^j + \xi_i^kg^j\partial_k \nabla D^{\alpha-\alpha'} \xi_i^j\bigg)
\end{align*}
such that
\begin{equation} \nonumber (\ref{theactualcommutator}) = \sum_{j=1}^3\sum_{k=1}^3\bigg(D^{\alpha-\alpha'}\xi_i^j\partial_j\xi_i^k\partial_kg - \xi_i^k\partial_kD^{\alpha-\alpha'} \xi_i^j \partial_jg - \xi_i^kg^j\partial_k \nabla D^{\alpha-\alpha'} \xi_i^j\bigg).\end{equation}
Due to this expression, we readily obtain that
$$\inner{D^{\alpha}f}{\big(B_{D^{\alpha -\alpha'} \xi_i}\mathcal{L}_{\xi_i} - \mathcal{L}_{\xi_i}B_{D^{\alpha-\alpha'} \xi_i}\big)D^{\alpha'}f} \leq c\norm{\xi_i}_{W^{\theta + 2,\infty}}^2\norm{f}_{W^{\theta,2}}^2.$$
In total then, we have verified that
$$(\ref{rewriteofprime}) \leq  c\norm{\xi_i}_{W^{\theta + 2,\infty}}^2\norm{f}_{W^{\theta,2}}^2$$
and as a consequence the same bound on (\ref{third term}), from which we deduce the first inequality of the Proposition. For the second inequality and to conclude the proof, using (\ref{thenevidently}) once more, we see that
    \begin{align*}
        \inner{D^\alpha B_i f}{D^\beta f} &= \Big \langle \sum_{\alpha' < \alpha}B_{D^{\alpha-\alpha'}\xi_i}D^{\alpha'}f + B_{\xi_i}D^\alpha f, D^\beta f \Big\rangle\\
        &= \Big \langle \sum_{\alpha' < \alpha}B_{D^{\alpha-\alpha'}\xi_i}D^{\alpha'}f , D^\beta f \Big\rangle + \inner{\mathcal{L}_{\xi_i}D^\alpha f}{D^\beta f} + \inner{\mathcal{T}_{\xi_i}D^\alpha f}{D^\beta f}
    \end{align*}
therefore
\begin{align*}
    &\inner{D^\alpha B_i f}{D^\beta f} + \inner{D^\beta B_i f}{D^\alpha f}\\
    &= \Big \langle \sum_{\alpha' < \alpha}B_{D^{\alpha-\alpha'}\xi_i}D^{\alpha'}f , D^\beta f \Big\rangle + \Big \langle \sum_{\beta' < \beta}B_{D^{\beta-\beta'}\xi_i}D^{\beta'}f , D^\alpha f \Big\rangle + \inner{\mathcal{T}_{\xi_i}D^\alpha f}{D^\beta f} + \inner{\mathcal{T}_{\xi_i}D^\beta f}{D^\alpha f} \\& \quad +  \inner{\mathcal{L}_{\xi_i}D^\alpha f}{D^\beta f} + \inner{\mathcal{L}_{\xi_i}D^\beta f}{D^\alpha f}.
\end{align*}
In the bottom line we observe complete cancellation,
$$\inner{\mathcal{L}_{\xi_i}D^\alpha f}{D^\beta f} + \inner{\mathcal{L}_{\xi_i}D^\beta f}{D^\alpha f} = -\inner{D^\alpha f}{\mathcal{L}_{\xi_i}D^\beta f} + \inner{\mathcal{L}_{\xi_i}D^\beta f}{D^\alpha f} = 0.$$
The remaining terms are again constructed to be of order $\theta$, such that Cauchy-Schwarz and (\ref{boundsonB_i}) yield the result. 

\end{proof}

We are now set up to prove Proposition \ref{transport and stretching}.

\begin{proof}[Proof of Proposition \ref{transport and stretching}:]
    Let us first consider (\ref{bigbound1}), assuming for now that $m$ is even and using the representation (\ref{even power inner product}), we have that
    \begin{align*}
        \inner{\mathcal{P}B_i^2f}{f}_{A^{\frac{m}{2}}} &= \sum_{k_1=1}^3 \dots \sum_{k_{\frac{m}{2}}=1}^3\sum_{l_1=1}^3 \dots \sum_{l_{\frac{m}{2}}=1}^3\inner{\partial_{k_1}^2 \dots \partial_{k_{\frac{m}{2}}}^2\mathcal{P}B_i^2f}{\partial_{l_1}^2 \dots \partial_{l_{\frac{m}{2}}}^2f}\\
        &= \sum_{k_1=1}^3 \dots \sum_{k_{\frac{m}{2}}=1}^3\sum_{l_1=1}^3 \dots \sum_{l_{\frac{m}{2}}=1}^3\inner{\partial_{k_1}^2 \dots \partial_{k_{\frac{m}{2}}}^2 B_i^2f}{\partial_{l_1}^2 \dots \partial_{l_{\frac{m}{2}}}^2f}
    \end{align*}
    having commuted $\mathcal{P}$ with derivatives, taken it to the other side, commuted through the derivatives again and absorbed into $f$. For the other term,
    \begin{align*}
        \norm{\mathcal{P}B_if}_{A^{\frac{m}{2}}}^2 = \norm{\mathcal{P}\Delta^{\frac{m}{2}}\mathcal{P}B_if}^2 = \norm{\mathcal{P}\Delta^{\frac{m}{2}}B_if}^2 \leq \norm{\Delta^{\frac{m}{2}}B_if}^2
    \end{align*}
    having used that $\mathcal{P}$ is bounded on $L^2\left(\mathbb{T}^3;\R^3\right)$. Expressing this as in (\ref{even power inner product}), we ultimately have that
    \begin{align*}
        &\inner{\mathcal{P}B_i^2f}{f}_{A^{\frac{m}{2}}} + \norm{\mathcal{P}B_if}_{A^{\frac{m}{2}}}^2\\
        &\leq \sum_{k_1=1}^3 \dots \sum_{k_{\frac{m}{2}}=1}^3\sum_{l_1=1}^3 \dots \sum_{l_{\frac{m}{2}}=1}^3\left(\inner{\partial_{k_1}^2 \dots \partial_{k_{\frac{m}{2}}}^2 B_i^2f}{\partial_{l_1}^2 \dots \partial_{l_{\frac{m}{2}}}^2f} + \inner{\partial_{k_1}^2 \dots \partial_{k_{\frac{m}{2}}}^2 B_if}{\partial_{l_1}^2 \dots \partial_{l_{\frac{m}{2}}}^2B_if} \right)
    \end{align*}
which admits a bound from (\ref{otherbound1}) as each term in the summand is of the form $\inner{D^\alpha B_i^2f}{D^\beta f} +  \inner{D^\alpha B_if}{D^\beta B_if}$, and can be paired with a corresponding term $\inner{D^\beta B_i^2f}{D^\alpha f} +  \inner{D^\beta B_if}{D^\alpha B_if}$. Where $m$ is odd the first term is simplified similarly with (\ref{odd power inner product}), and in the second term we use from (\ref{odd power inner product}) that
\begin{align*}
    \norm{\mathcal{P}B_if}_{A^{\frac{m}{2}}}^2 = \sum_{j=1}^3\norm{\partial_jA^{\frac{m-1}{2}}\mathcal{P}B_if}^2 =  \sum_{j=1}^3\norm{\mathcal{P}\partial_j\Delta^{\frac{m-1}{2}}B_if}^2 \leq \sum_{j=1}^3\norm{\partial_j\Delta^{\frac{m-1}{2}}B_if}^2.
\end{align*}
The control now follows as in the even case, and we conclude (\ref{bigbound1}). The second inequality (\ref{bigbound2}) similarly holds as a consequence of (\ref{otherbound2}), using the representations (\ref{odd power inner product}), (\ref{even power inner product}) and passing $\mathcal{P}$ onto $f$.

\end{proof}

Estimates in the usual Sobolev norm are given below.

\begin{proposition} \label{transport and stretching 2}
    Fix $m \in \N$. There exists a constant $c$ such that for all $\xi_i \in W^{m+2,\infty} \cap L^2_{\sigma}$ and $f \in W^{m+2,2}_{\sigma}$,
    \begin{align*}
   \inner{\mathcal{P}B_i^2f}{f}_{W^{m,2}} +  \norm{\mathcal{P}B_if}_{W^{m,2}}^2  &\leq c\norm{\xi_i}_{W^{m +2,\infty}}^2\norm{f}_{W^{m,2}}^2 ,\\
    \inner{\mathcal{P}B_if}{f}_{W^{m,2}}^2 &\leq c\norm{\xi_i}^2_{W^{m +1,\infty}}\norm{f}^4_{W^{m,2}}.
\end{align*}
\end{proposition}

\begin{proof}
    The proof is now entirely contained in what has come before, treating the Leray Projector as in Proposition \ref{transport and stretching} and simply taking $\alpha=\beta$ in Lemma \ref{lemma for conservation B_i}. 
\end{proof}

\subsection{Transport Noise} \label{subs transport noise estimates}

We follow the same structure as the previous subsection, now for the estimates on transport noise.

\begin{proposition} \label{first transport only}
        Fix $m \in \N$. There exists a constant $c$ such that for all $\xi_i \in W^{m+2,\infty} \cap L^2_{\sigma}$ and $f \in W^{m+2,2}_{\sigma}$,
    \begin{align}
   \inner{\left(\mathcal{P}\mathcal{L}_{\xi_i}\right)^2f}{f}_{A^{\frac{m}{2}}} +  \norm{\mathcal{P}\mathcal{L}_{\xi_i}f}_{A^{\frac{m}{2}}}^2  &\leq c\norm{\xi_i}_{W^{m +2,\infty}}^2\norm{f}_{A^{\frac{m}{2}}}^2 \label{bigbound1Li} ,\\
    \inner{\mathcal{P}\mathcal{L}_{\xi_i}f}{f}_{A^{\frac{m}{2}}}^2 &\leq c\norm{\xi_i}^2_{W^{m,\infty}}\norm{f}^4_{A^{\frac{m}{2}}}. \label{bigbound2Li}
\end{align}
\end{proposition}

\begin{proof}
    Let us first consider (\ref{bigbound1Li}), which along the lines of Proposition \ref{transport and stretching} will be demonstrated by obtaining a control on 
\begin{equation}
    \inner{D^\alpha \mathcal{L}_{\xi_i}\mathcal{P}\mathcal{L}_{\xi_i}f}{D^\beta f} +  \inner{D^\alpha \mathcal{P}\mathcal{L}_{\xi_i}f}{D^\beta \mathcal{L}_{\xi_i}f} + \inner{D^\beta\mathcal{L}_{\xi_i}\mathcal{P}\mathcal{L}_{\xi_i}f}{D^\alpha f} +  \inner{D^\beta \mathcal{P}\mathcal{L}_{\xi_i}f}{D^\alpha\mathcal{L}_{\xi_i}f}  \label{Lotherbound1}
\end{equation}
similarly to Lemma \ref{lemma for conservation B_i}. On this occasion the Leray Projector persists given that it is stuck between the two $\mathcal{L}_{\xi_i}$ in the leftmost term so cannot be passed off as we have used elsewhere. For this reason we chose to maintain $\mathcal{P}$ in the norm term, allowing us to match with the previously discussed term in looking for cancellation. To control (\ref{Lotherbound1}) we follow Lemma \ref{lemma for conservation B_i}, reducing (\ref{Lotherbound1}) to the sum of a first expression
\begin{align}
  \nonumber  \Big\langle \mathcal{P}\sum_{\alpha' < \alpha} \mathcal{L}_{D^{\alpha-\alpha'} \xi_i}D^{\alpha'}f &, \sum_{\beta' < \beta} \mathcal{L}_{D^{\beta-\beta'} \xi_i}D^{\beta'}f\Big\rangle\\ &+ \Big\langle \mathcal{P}\sum_{\beta' < \beta} \mathcal{L}_{D^{\beta-\beta'} \xi_i}D^{\beta'}f , \sum_{\alpha' < \alpha} \mathcal{L}_{D^{\alpha-\alpha'} \xi_i}D^{\alpha'}f\Big\rangle \label{second term L}
\end{align}
corresponding to (\ref{second term}), and a second expression
\begin{align}
    \nonumber \sum_{\alpha' < \alpha}\inner{\mathcal{L}_{D^{\alpha-\alpha'} \xi_i}D^{\alpha'}\mathcal{P}\mathcal{L}_{\xi_i}f}{D^\beta f} &+ \sum_{\beta' < \beta}\inner{\mathcal{P}\mathcal{L}_{\xi_i}D^\alpha f}{\mathcal{L}_{D^{\beta-\beta'} \xi_i}D^{\beta'}f} \\&+ \sum_{\beta' < \beta}\inner{\mathcal{L}_{D^{\beta-\beta'} \xi_i}D^{\beta'}\mathcal{P}\mathcal{L}_{\xi_i}f}{D^\alpha f} + \sum_{\alpha' < \alpha}\inner{\mathcal{P}\mathcal{L}_{\xi_i}D^\beta f}{\mathcal{L}_{D^{\alpha-\alpha'} \xi_i}D^{\alpha'}f} \label{third term L}
\end{align}
corresponding to (\ref{third term}). Note here there that is no corresponding term to (\ref{first term}) due to the absence of the additional stretching $\mathcal{T}$. Like its counterpart (\ref{second term}), (\ref{second term L}) is constructed to be of order $m$ in all terms hence is bounded directly. The difficulties arise in (\ref{third term L}), which boils down to
\begin{align*}
    \sum_{\alpha' < \alpha}\left( \inner{D^\beta f}{\mathcal{L}_{D^{\alpha-\alpha'} \xi_i}D^{\alpha'}\mathcal{P}\mathcal{L}_{\xi_i}f} + \inner{D^\beta f}{\mathcal{L}_{\xi_i}^*\mathcal{P}\mathcal{L}_{D^{\alpha-\alpha'} \xi_i}D^{\alpha'}f}\right) 
\end{align*}
as well as a second sum over the reversed indices. We expand out $D^{\alpha'}\mathcal{P}\mathcal{L}_{\xi_i}f$ and separate into

\begin{align*}& \Big\langle D^\beta f,\mathcal{L}_{D^{\alpha-\alpha'}\xi_i}\mathcal{P}\bigg(\sum_{\gamma < \alpha'}\mathcal{L}_{D^{\alpha'-\gamma}\xi_i}D^\gamma f + \mathcal{L}_{\xi_i}D^{\alpha'}f\bigg)+ \mathcal{L}^*_{\xi_i}\mathcal{P}\mathcal{L}_{D^{\alpha-\alpha'} \xi_i}D^{\alpha'}f\Big\rangle\\
= & \Big\langle D^\beta f,\sum_{\gamma < \alpha'}\mathcal{L}_{D^{\alpha-\alpha'}\xi_i}\mathcal{P}\mathcal{L}_{D^{\alpha'-\gamma}\xi_i}D^\gamma f\Big\rangle + \inner{D^\beta f}{\mathcal{L}_{D^{\alpha-\alpha'}\xi_i}\mathcal{P}\mathcal{L}_{\xi_i}D^{\alpha'}f + \mathcal{L}^*_{\xi_i}\mathcal{P}\mathcal{L}_{D^{\alpha-\alpha'} \xi_i}D^{\alpha'}f}.  \end{align*}
The first inner product has once more removed the top order derivative and is directly controllable. In the second inner product we need to obtain some cancellation of the top order derivative, but in the presence of $\mathcal{P}$ we are not left with a clear commutator as we were in Lemma \ref{lemma for conservation B_i}. We resolve this issue by simply calling upon what worked before, adding and subtracting the relevant $\mathcal{T}$ terms to introduce $B$ which plays well with the Leray Projector. Precisely we obtain that
\begin{align*}
   &\inner{D^\beta f}{\mathcal{L}_{D^{\alpha-\alpha'}\xi_i}\mathcal{P}\mathcal{L}_{\xi_i}D^{\alpha'}f + \mathcal{L}^*_{\xi_i}\mathcal{P}\mathcal{L}_{D^{\alpha-\alpha'} \xi_i}D^{\alpha'}f}\\ & \qquad \qquad \qquad = \inner{D^\beta f}{\left(B_{D^{\alpha-\alpha'}\xi_i} - \mathcal{T}_{D^{\alpha-\alpha'}\xi_i} \right)\mathcal{P}\mathcal{L}_{\xi_i}D^{\alpha'}f - \left(B_{\xi_i} - \mathcal{T}_{\xi_i} \right)\mathcal{P}\mathcal{L}_{D^{\alpha-\alpha'} \xi_i}D^{\alpha'}f}.
\end{align*}
The resulting terms involving $\mathcal{T}$ are again of correct order, so we are only concerned with those involving $B$. In the first we rewrite
\begin{align*}
    \inner{D^\beta f}{B_{D^{\alpha-\alpha'}\xi_i}\mathcal{P}\mathcal{L}_{\xi_i}D^{\alpha'}f} &=  \inner{D^\beta f}{\mathcal{P}B_{D^{\alpha-\alpha'}\xi_i}\mathcal{P}\mathcal{L}_{\xi_i}D^{\alpha'}f}\\
    &= \inner{D^\beta f}{\mathcal{P}B_{D^{\alpha-\alpha'}\xi_i}\mathcal{L}_{\xi_i}D^{\alpha'}f}\\
    &= \inner{D^\beta f}{B_{D^{\alpha-\alpha'}\xi_i}\mathcal{L}_{\xi_i}D^{\alpha'}f}
\end{align*}
having used that $\mathcal{P}B_{D^{\alpha-\alpha'}\xi_i}\mathcal{P} = \mathcal{P}B_{D^{\alpha-\alpha'}\xi_i}$, as $D^{\alpha-\alpha'}\xi_i \in L^2_{\sigma} \cap W^{1,\infty}$ so enjoys the regularity that allowed for the conclusion of $\mathcal{P}B_i\mathcal{P} = \mathcal{P}B_i$. Repeating this procedure, we have that 
\begin{align*}
\inner{D^\beta f}{B_{D^{\alpha-\alpha'}\xi_i} \mathcal{P}\mathcal{L}_{\xi_i}D^{\alpha'}f - B_{\xi_i} \mathcal{P}\mathcal{L}_{D^{\alpha-\alpha'} \xi_i}D^{\alpha'}f}
 =  \inner{D^\beta f}{B_{D^{\alpha-\alpha'}\xi_i}\mathcal{L}_{\xi_i}D^{\alpha'}f - B_{\xi_i}\mathcal{L}_{D^{\alpha-\alpha'} \xi_i}D^{\alpha'}f}.
\end{align*}
Now we can once more decompose $B$ into its $\mathcal{L}$ and $\mathcal{T}$ terms, happy to control those with $\mathcal{T}$ directly, leaving the final problematic term as
$$\inner{D^\beta f}{\mathcal{L}_{D^{\alpha-\alpha'}\xi_i}\mathcal{L}_{\xi_i}D^{\alpha'}f - \mathcal{L}_{\xi_i}\mathcal{L}_{D^{\alpha-\alpha'} \xi_i}D^{\alpha'}f} $$
where the commutator was calculated and controlled explicitly in (\ref{theactualcommutator}), with which we conclude the justification of (\ref{bigbound1Li}). Furthermore, (\ref{bigbound2Li}) requires nothing beyond the proof of (\ref{bigbound2}). 
    
\end{proof}

Analogously to Proposition \ref{transport and stretching 2}, we also have the following.

\begin{proposition} \label{transport only 2}
    Fix $m \in \N$. There exists a constant $c$ such that for all $\xi_i \in W^{m+2,\infty} \cap L^2_{\sigma}$ and $f \in W^{m+2,2}_{\sigma}$,
    \begin{align*}
   \inner{\left(\mathcal{P}\mathcal{L}_{\xi_i}\right)^2f}{f}_{W^{m,2}} +  \norm{\mathcal{P}\mathcal{L}_{\xi_i}f}_{W^{m,2}}^2  &\leq c\norm{\xi_i}_{W^{m +2,\infty}}^2\norm{f}_{W^{m,2}}^2 ,\\
    \inner{\mathcal{P}\mathcal{L}_{\xi_i}f}{f}_{W^{m,2}}^2 &\leq c\norm{\xi_i}^2_{W^{m,\infty}}\norm{f}^4_{W^{m,2}}.
\end{align*}
\end{proposition}

\section{Strong Solutions of the Stochastic Euler Equation} \label{section strong solutions euler}

This section concerns the proof of Theorem \ref{main existence for euler}. Solutions are constructed from an inviscid limit of the corresponding stochastic Navier-Stokes equations. Our analysis begins by considering a smooth initial condition, uniformly bounded over $\Omega$. Uniform in viscosity estimates for the approximating sequence of stochastic Navier-Stokes equations are shown in Subsection \ref{subsection uniform estimates}. We show the existence of a limiting process and stopping time by a Cauchy approach in the spirit of [\cite{glatt2009strong}], which was extended in [\cite{goodair2024weak}]. The abstract result is given as Proposition \ref{amazing cauchy lemma} in the appendix, and the conditions to apply it are verified in Subsections \ref{subs cauchy} and \ref{subs weak equi}. Passage to the limit and existence of local smooth solutions is given in Subsection \ref{subs smooth}, where smoothness of the initial condition is optimised in Subsection \ref{subs initial condition relax}. Uniqueness is demonstrated in Subsection \ref{subs unique}, followed by the most intricate arguments of this section, for the maximality and blow-up, in Subsection \ref{subs maximality}. The proof of Theorem \ref{main existence for euler} concludes in Subsection \ref{subsection final steps} by extending to the case of an unbounded initial condition and verifying the Stratonovich identity. 

\subsection{Uniform Estimates for the Stochastic Navier-Stokes Equation} \label{subsection uniform estimates}

 In the direction of taking viscosity to zero in the stochastic Navier-Stokes equations, here we consider a viscosity $\nu_n$ and corresponding solutions $u^n$ to the equation 
\begin{equation} \label{projected Ito general noise}
    u^n_t = u_0 - \int_0^t\mathcal{P}\mathcal{L}_{u^n_s}u^n_s\ ds - \nu_n\int_0^t A u^n_s\, ds + \frac{1}{2}\int_0^t\sum_{i=1}^\infty \left(\mathcal{P}\mathcal{G}_i\right)^2u^n_s ds + \int_0^t \mathcal{P}\mathcal{G}u^n_s d\mathcal{W}_s 
\end{equation}
where $(\nu_n)$ converges to zero. Solution theory for these equations are established below.


\begin{proposition} \label{navier stokes existence prop}
 For $m \geq 6$ let $u_0 \in L^\infty\left( \Omega ; W^{m,2}_{\sigma} \right)$ be $\mathcal{F}_0-$measurable and each $\xi_i \in L^2_{\sigma} \cap  W^{m+3,\infty}$ such that $\sum_{i=1}^\infty \norm{\xi_i}_{W^{m+2,\infty}}^2 < \infty$. Fix any $M > 1$. There exists a pair $(u^n,\tau^M_n)$ where $u^n$ is a process such that for $\mathbbm{P}-a.e.$ $\omega$, $u_{\cdot}(\omega) \in C\left([0,T];W^{m,2}_{\sigma}\right)$, and $\tau^M_n$ is the $\mathbbm{P}-a.s.$ positive stopping time defined by 
 \begin{equation} \label{taumt}
        \tau^{M}_n := T \wedge \inf\left\{s \geq 0: \sup_{r \in [0,s]}\norm{u^n_r}_{W^{m-3,2}}^2dr \geq M + \norm{u_0}_{W^{m-3,2}}^2 \right\}
    \end{equation}
which satisfies:
\begin{enumerate}
    \item Existence as a local strong solution, in the sense that $u_{\cdot}(\omega)\mathbbm{1}_{\cdot \leq \tau^M_n(\omega)} \in L^2\left([0,T];W^{m+1,2}_{\sigma}\right)$ and with $u_{\cdot}\mathbbm{1}_{\cdot \leq \tau^M_n}$ progressively measurable in $W^{m+1,2}_{\sigma}$, satisfying (\ref{projected Ito general noise}) stopped at $\tau^M_n$ $\mathbbm{P}-a.s.$ in $L^2_{\sigma}$ for all $t \in [0,T]$.

    \item Uniqueness, meaning that if $(v, \gamma)$ was any other such local strong solution then  \begin{equation} \nonumber\mathbbm{P}\left(\left\{\omega \in \Omega: u^n_{t}(\omega) =  v_{t}(\omega)  \quad \forall t \in [0,\tau^M_n \wedge \gamma] \right\} \right) = 1. \end{equation}

    \item Uniform in viscosity estimates, whereby there exists a constant $C_M$ independent of $n$ (but dependent on $M$, $u_0$) such that \begin{equation} \label{uniform in viscosity bound}
\mathbbm{E}\left( \sup_{r \in [0,\tau^M_n]}\norm{u^n_r}_{W^{m-1,2}}^2 \right) \leq C_M. 
\end{equation}
    
\end{enumerate}

\end{proposition}

\begin{proof}
The existence and uniqueness items above for $\mathcal{G}_i = B_i$ were proven in [\cite{goodair2024high}] Proposition 3.7, as iterated applications of [\cite{goodair2024improved}] Theorem 2.9. More precisely a maximal strong solution is constructed where the maximal time is given by the blow-up in $C\left([0,\cdot];W^{1,2}_{\sigma} \right) \cap L^2\left([0,\cdot];W^{2,2}_{\sigma} \right)$, a coarser norm than $C\left([0,\cdot];W^{m-3,2}_{\sigma}\right)$ until $T$, therefore $\tau^M_n$ is strictly less than the maximal time and ensures that the maximal solution is a local one up until $\tau^M_n$. Thus the local strong solution given by the maximal solution stopped at the local time, $(u^n_{\cdot \wedge \tau^{M}_n}, \tau^M_n)$ enjoys the desired regularity from [\cite{goodair2024high}] Proposition 3.7. With the bounds demonstrated for transport noise in Proposition \ref{first transport only} then the result holds similarly for $\mathcal{G}_i = \mathcal{P}\mathcal{L}_{\xi_i}$. It remains to demonstrate the uniform in viscosity estimate. To rigorously take expectations in verifying this estimate we introduce the stopping time
\begin{equation} \label{alpha R} \alpha_R:= \tau^M_n \wedge \inf\left\{s \geq 0: \sup_{r \in [0,s]}\norm{u^n_r}_{W^{m,2}}^2 + \int_0^{s}\norm{u^n_r}_{W^{m+1,2}}^2dr \geq R \right\}\end{equation}
along with notation
\begin{equation} \label{check notation} \check{u}^n_{\cdot}:= u^n_{\cdot}\mathbbm{1}_{\cdot \leq \alpha_R}. \end{equation}
Stopping $u^n$ at $\alpha_R$ and applying the It\^{o} Formula, appreciating that the identity of (\ref{projected Ito general noise}) is in fact satisfied in $W^{m-1,2}_{\sigma}$, we deduce the energy equality
\begin{align}\nonumber\norm{u^{n}_{ r \wedge \alpha_R}}_{W^{m-1,2}}^2 &= \norm{u_0}_{W^{m-1,2}}^2 - 2\int_0^{r\wedge \alpha_R}\inner{\mathcal{P}\mathcal{L}_{u^n_s}u^n_s}{u^n_s}_{W^{m-1,2}}ds - 2\nu_n\int_0^{r\wedge \alpha_R}\inner{Au^n_s}{u^n_s}_{W^{m-1,2}}ds\\ \nonumber &+ \int_0^{r\wedge \alpha_R} \sum_{i=1}^\infty \inner{\left(\mathcal{P}\mathcal{G}_i\right)^2u^n_s}{u^n_s}_{W^{m-1,2}}ds + \int_0^{r\wedge\alpha_R}\sum_{i=1}^\infty\norm{\mathcal{P}\mathcal{G}_iu^n_s}_{W^{m-1,2}}^2ds\\ &+ 2\sum_{i=1}^\infty\int_0^{r\wedge\alpha_R}\inner{\mathcal{P}\mathcal{G}_iu^n_s}{u^n_s}_{W^{m-1,2}}dW^i_s\nonumber.
\end{align}
Here we look to reduce the terms. Pertaining to the Stokes Operator we observe that \begin{equation} \label{Stokes is non neg} -\inner{Au^n_s}{u^n_s}_{W^{m-1,2}} = \inner{\mathcal{P}\Delta u^n_s}{u^n_s}_{W^{m-1,2}} = -\sum_{j=1}^3\norm{\partial_ju^n_s}_{W^{m-1,2}}^2,\end{equation} which is negative, whilst for the nonlinear term we recall the control (\ref{nonlinear term control}) to verify
$$\left\vert \inner{\mathcal{P}\mathcal{L}_{u^n_s}u^n_s}{u^n_s}_{W^{m-1,2}} \right\vert \leq c \norm{u^n_s}_{W^{1,\infty}}\norm{u^n_s}_{W^{m-1,2}}^2.$$
The following two terms, attributed to the It\^{o}-Stratonovich corrector and the quadratic variation, are controlled in the two cases of $\mathcal{G}$ by Propositions \ref{transport and stretching 2} and \ref{transport only 2}. Inserting these bounds into the given energy equality, we furthermore take absolute values and the supremum in time. Appreciating that for any $t \in [0,T]$, $\sup_{r \in [0,t]}\norm{\check{u}^n_r}_{W^{m-1,2}}^2 = \sup_{r \in [0,t]}\norm{u^n_{r \wedge \alpha_R}}_{W^{m-1,2}}^2$ as the hitting threshold is met continuously, taking expectation, applying the Burkholder-Davis-Gundy Inequality and recasting into the $\check{u}^n$ notation of (\ref{check notation}), we deduce that
\begin{align}\nonumber \mathbbm{E}\left(\sup_{r \in [0,t]}\norm{\check{u}^{n}_{ r}}_{W^{m-1,2}}^2\right) &\leq c\mathbbm{E}\left(\norm{u_0}_{W^{m-1,2}}^2 + \int_0^{t}\left[1 + \norm{\check{u}^n_s}_{W^{1,\infty}} \right]\norm{\check{u^n_s}}_{W^{m-1,2}}^2ds\right)\\  &+ c\mathbbm{E}\left(\int_0^{t}\sum_{i=1}^\infty\inner{\mathcal{P}\mathcal{G}_i\check{u}^n_s}{\check{u}^n_s}_{W^{m-1,2}}^2ds\right)^{\frac{1}{2}}\label{to be inserted}.
\end{align}
In the final term we invoke Proposition \ref{transport and stretching 2} or \ref{transport only 2} as relevant, amounting to
\begin{align} \nonumber
c\mathbbm{E}\left(\int_0^{t}\sum_{i=1}^\infty\inner{\mathcal{P}\mathcal{G}_i\check{u}^n_s}{\check{u}^n_s}_{W^{m-1,2}}^2ds\right)^{\frac{1}{2}} &\leq c\mathbbm{E}\left(\int_0^{t}\norm{\check{u}^n_s}_{W^{m-1,2}}^4ds\right)^{\frac{1}{2}}\\ \nonumber
&\leq c\mathbbm{E}\left(\sup_{r \in [0,t]}\norm{\check{u^n_r}}_{W^{m-1,2}}^2\int_0^{t}\norm{\check{u}^n_s}_{W^{m-1,2}}^2 ds\right)^{\frac{1}{2}}\\ \label{BDG type control}
&\leq \frac{1}{2}\mathbbm{E}\left(\sup_{r \in [0,t]}\norm{\check{u^n_r}}_{W^{m-1,2}}^2\right) + c\mathbbm{E}\left(\int_0^{t}\norm{\check{u}^n_s}_{W^{m-1,2}}^2 ds\right)
\end{align}
having used Young's Inequality in the final line. Inserting this back into (\ref{to be inserted}) grants us that
\begin{align}\label{granted} \mathbbm{E}\left(\sup_{r \in [0,t]}\norm{\check{u}^{n}_{ r}}_{W^{m-1,2}}^2\right) &\leq c\mathbbm{E}\left(\norm{u_0}_{W^{m-1,2}}^2 + \int_0^{t}\left[1 + \norm{\check{u}^n_s}_{W^{1,\infty}} \right]\norm{\check{u^n_s}}_{W^{m-1,2}}^2ds\right)
\end{align}
where we have simply multiplied both sides by $2$. Here we utilise the control on $\check{u}^n$ afforded to us by $\tau^M_n$, defined in (\ref{taumt}), as well as the fact that $u_0 \in L^\infty\left(\Omega; W^{m,2}_{\sigma} \right)$ so $\sup_{r \in [0,t]}\norm{\check{u}^n_r}_{W^{m-3,2}}^2 \leq C_M$ uniformly in $\omega$. Due to the Sobolev Embedding of $W^{m-3,2} \xhookrightarrow{} W^{1,\infty}$ then we likewise have that $\sup_{r \in [0,t]}\norm{\check{u}^n_r}_{W^{1,\infty}}^2 \leq C_M$. Therefore,
\begin{align}\nonumber \mathbbm{E}\left(\sup_{r \in [0,t]}\norm{\check{u}^{n}_{ r}}_{W^{m-1,2}}^2 \right) &\leq C_M\left(1 + \mathbbm{E}\int_0^{t}\norm{\check{u^n_s}}_{W^{m-1,2}}^2ds\right)
\end{align}
from which the standard Gr\"{o}nwall Inequality yields that
$$  \mathbbm{E}\left(\sup_{r \in [0,T]}\norm{\check{u}^{n}_{ r}}_{W^{m-1,2}}^2 \right) \leq C_M$$
which we write again explicitly in terms of $\alpha_R$ by
$$  \mathbbm{E}\left(\sup_{r \in [0,\alpha_R]}\norm{u^{n}_{ r}}_{W^{m-1,2}}^2\right) \leq C_M$$
whilst noting that $C_M$ is independent of $R$. We use that $\alpha_R$ is $\mathbbm{P}-a.s.$ monotonically increasing to $\tau^M_n$ as $R \rightarrow \infty$, so continuity of the integrand and applying the Monotone Convergence Theorem obtains the result.
\end{proof}

\subsection{Cauchy Property} \label{subs cauchy}


\begin{proposition} \label{cauchy prop}
    We have that
    $$\lim_{j \rightarrow \infty}\sup_{n \geq j} \mathbbm{E} \left( \sup_{r \in [0,\tau^M_n \wedge \tau^M_j]}\norm{u^n_r - u^j_r}_{W^{m-3,2}}^2 \right) = 0. $$
\end{proposition}

\begin{proof}
    Our approach is similar to the proof of Proposition \ref{navier stokes existence prop}, where we shall look at the energy of the identity satisfied by $u^n-u^j$ for $n \geq j$. To this end we define $\alpha := \tau^M_n \wedge \tau^M_j$, noting that we do not need to truncate by $R$ as in the previous result given the uniform estimate of (\ref{uniform in viscosity bound}). We obtain the energy identity
\begin{align}\nonumber\norm{u^{n}_{ r \wedge \alpha} - u^{j}_{ r \wedge \alpha}}_{W^{m-3,2}}^2 &= - 2\int_0^{r\wedge \alpha}\inner{\mathcal{P}\mathcal{L}_{u^n_s}u^n_s - \mathcal{P}\mathcal{L}_{u^j_s}u^j_s}{u^n_s - u^j_s}_{W^{m-3,2}}ds\\ \nonumber &- 2\nu_n\int_0^{r\wedge \alpha}\inner{Au^n_s}{u^n_s - u^j_s}_{W^{m-3,2}}ds + 2\nu_j\int_0^{r\wedge \alpha}\inner{Au^j_s}{u^n_s - u^j_s}_{W^{m-3,2}}ds\\ \nonumber &+ \int_0^{r\wedge \alpha} \sum_{i=1}^\infty \left( \inner{\left(\mathcal{P}\mathcal{G}_i\right)^2\left(u^n_s - u^j_s\right)}{u^n_s-u^j_s}_{W^{m-3,2}} + \norm{\mathcal{P}\mathcal{G}_i\left(u^n_s - u^j_s\right)}_{W^{m-3,2}}^2 \right) ds\\ &+ 2\sum_{i=1}^\infty\int_0^{r\wedge\alpha}\inner{\mathcal{P}\mathcal{G}_i\left(u^n_s - u^j_s\right)}{u^n_s - u^j_s}_{W^{m-3,2}}dW^i_s\nonumber
\end{align}
and look to reduce the terms individually. For the nonlinear term, exactly as in [\cite{glatt2009strong}] (7.12) we obtain that
\begin{align*}
    &\left\vert\inner{\mathcal{P}\mathcal{L}_{u^n_s}u^n_s - \mathcal{P}\mathcal{L}_{u^j_s}u^j_s}{u^n_s - u^j_s}_{W^{m-3,2}}\right\vert\\ & \qquad  \leq c\norm{u^n_s}_{W^{m-2,2}}^2\norm{u^n_s-u^j_s}_{W^{m-4,2}}^2 + c\left(1 + \norm{u^n_s}_{W^{m-3,2}} + \norm{u^j_s}_{W^{m-3,2}} \right)\norm{u^n_s - u^j_s}_{W^{m-3,2}}^2.
\end{align*}
In the Stokes terms we may use a coarse control with Cauchy-Schwarz and Young's Inequality, bounding them both by
\begin{align} \label{hokey stokey bound}
     c(\nu_n + \nu_j)\int_0^{r\wedge \alpha}\norm{u^n_s}_{W^{m-1,2}}^2 + \norm{u^j_s}_{W^{m-1,2}}^2ds.
\end{align}
Owing to the linearity of the noise terms, their treatment requires no difference to that conducted in the proof of Proposition \ref{navier stokes existence prop}. Indeed by repeating the steps of that proof, with the new notation $\hat{u}^n_{\cdot} = u^n_{\cdot}\mathbbm{1}_{\cdot \leq \alpha}$, $\hat{u}^j_{\cdot} = u^j_{\cdot}\mathbbm{1}_{\cdot \leq \alpha}$, we achieve that
\begin{align}\nonumber &\mathbbm{E}\left(\sup_{r \in [0,t]}\norm{\hat{u}^{n}_{r} - \hat{u}^j_r }_{W^{m-3,2}}^2\right)\\ & \quad \leq c\mathbbm{E}\left(\int_0^{t}\left[1 + \norm{\hat{u}^n_s}_{W^{m-3,2}} + \norm{\hat{u}^j_s}_{W^{m-3,2}} \right]\norm{\hat{u}^n_s - \hat{u}^j_s}_{W^{m-3,2}}^2ds\right) \nonumber \\ \label{this one is numbered} & \quad + c(\nu_n + \nu_j)\mathbbm{E}\left(\int_0^{t} \norm{\hat{u}^n_s}_{W^{m-1,2}}^2 + \norm{\hat{u}^j_s}_{W^{m-1,2}}^2 ds\right) 
 +c \mathbbm{E}\left( \int_0^t \norm{\hat{u}^n_s}_{W^{m-2,2}}^2\norm{\hat{u}^n_s - \hat{u}^j_s}_{W^{m-4,2}}^2 ds \right).
\end{align}
In the first term we shall again use the control granted by the stopping time $\tau^M_n \wedge \tau^M_j$. In the second term we first employ that $\nu_n + \nu_j \leq 2\nu_j$, and furthermore  that
\begin{align*}c\nu_j\mathbbm{E}\left(\int_0^t \norm{\hat{u}^n_s}_{W^{m-1,2}}^2 + \norm{\hat{u}^j_s}_{W^{m-1,2}}^2ds \right) &= c\nu_j\int_0^t \mathbbm{E}\left(\norm{\hat{u}^n_s}_{W^{m-2,2}}^2 + \norm{\hat{u}^j_s}_{W^{m-1,2}}^2\right) ds\\ &\leq c\nu_j \int_0^t C_M ds \leq C_M \nu_j \end{align*}
due to (\ref{uniform in viscosity bound}), where $C_M$ here is also dependent on $T$ which remains fixed.  Therefore, revisiting (\ref{this one is numbered}), we have achieved that
\begin{align}\nonumber \mathbbm{E}\left(\sup_{r \in [0,t]}\norm{\hat{u}^{n}_{r} - \hat{u}^j_r }_{W^{m-3,2}}^2 \right) &\leq C_M\left[\nu_j + \mathbbm{E}\left(\int_0^{t}\norm{\hat{u}^n_s - \hat{u}^j_s}_{W^{m-3,2}}^2ds\right)\right]\\ \nonumber &+c \mathbbm{E}\left( \int_0^t \norm{\hat{u}^n_s}_{W^{m-2,2}}^2\norm{\hat{u}^n_s - \hat{u}^j_s}_{W^{m-4,2}}^2 ds \right).
\end{align}
With a coarse bound on the final term, allowing the generic constant $c$ to be dependent on $M$ and $T$ and applying the classical Gr\"onwall lemma, 
\begin{align}\nonumber &\mathbbm{E}\left(\sup_{r \in [0,T]}\norm{\hat{u}^{n}_{r} - \hat{u}^j_r }_{W^{m-3,2}}^2 \right) \leq c \left( \nu_j + \mathbbm{E}\left[ \sup_{r \in [0,T]}\left(\norm{\hat{u}^n_r}_{W^{m-2,2}}^2\norm{\hat{u}^n_r - \hat{u}^j_r}_{W^{m-4,2}}^2\right) \right] \right).
\end{align}
Moreover, the result will follow once we show that
\begin{align} \label{ultimate yield}
    \lim_{j \rightarrow \infty} \sup_{n\geq j} \mathbbm{E} \left[ \sup_{r\in [0,T]} \left(\norm{\hat{u}^n_r}_{W^{m-2,2}}^2
    \norm{\hat{u}^n_r - \hat{u}^j_r}_{W^{m-4,2}}^2\right)  \right] =0.
\end{align}
The following arguments will be notationally heavy, so we replace $\hat{u}^{n}$ by simply $u^{n}$ for convenience, although we stress that it remains truly truncated at $\alpha$. Furthermore let us introduce the notation $v_{t}=u^{n}_{t}-u^{j}_{t}$. To estimate this term we again aim to use the Gr\"onwall inequality, hence we shall determine an evolution equation for the product
$\norm{u^{n}_{\cdot}}_{W^{m-2,2}}^{2}\norm{v_{\cdot}}_{W^{m-4,2}}^{2}$. To this end we note that by the It\^o formula,
\begin{align*}
    d\norm{u^{n}_{t}}_{W^{m-2,2}}^{2} &= -2 \inner{\proj \nl{u^{n}_{t}}{u^{n}_{t}}}{u^{n}_{t}}_{W^{m-2,2}}dt \\
    &- 2\nu_{n}\inner{Au^{n}_{t}}{u^{n}_{t}}_{W^{m-2,2}} dt \\
    &+ \infsum{i}\left[\inner{\left(\proj\Gi\right)^{2} u^{n}_{t}}{u^{n}_{t}}_{W^{m-2,2}} + \norm{\proj \Gi u^{n}_{t}}_{W^{m-2,2}}^2\right]dt \\
    &+2 \infsum{i} \inner{\proj\Gi u^{n}_{t}}{u^{n}_{t}}_{W^{m-2,2}} dW_{t}^{i}\\
    &:= (I_{1} + I_{2} +I_{3}) dt + I_{4}d\mathcal{W}_{t}
\end{align*}
as well as
\begin{align*}
    d\norm{v_{t}}_{W^{m-4,2}}^{2} &= -2 \inner{\proj \nl{u^{n}_{t}}{u^{n}_{t}}-\proj \nl{u^{j}_{t}}{u^{j}_{t}}}{u^{n}_{t}-u^{j}_{t}}_{W^{m-4,2}}dt \\
    &- \left[2\nu_{n}\inner{Au^{n}_{t}}{u^{n}_{t}-u^{j}_{t}}_{W^{m-4,2}} + 2\nu_{j}\inner{Au^{j}_{t}}{u^{n}_{t}-u^{j}_{t}}_{W^{m-4,2}}\right]dt \\
    &+ \infsum{i}\left[\inner{\left(\proj\Gi\right)^{2} (u^{n}_{t}-u^{j}_{t})}{u^{n}_{t}-u^{j}_{t}}_{W^{m-4,2}} + \norm{\proj \Gi (u^{n}_{t}-u^{j}_{t})}_{W^{m-4,2}}^2\right]dt \\
    &+2 \infsum{i} \inner{\proj\Gi (u^{n}_{t}-u^{j}_{t})}{u^{n}_{t}-u^{j}_{t}}_{W^{m-4,2}} dW_{t}^{i}\\
    &:= (J_{1} + J_{2} +J_{3}) dt + J_{4}d\mathcal{W}_{t}.
\end{align*}
As the norms are simply real valued processes, then by the usual It\^o product rule we obtain
\begin{align*}
    d(\norm{u^{n}_{t}}_{W^{m-2,2}}^{2}\norm{v_{t}}_{W^{m-4,2}}^{2}) &= \norm{u^{n}_{t}}_{W^{m-2,2}}^{2}d\norm{v_{t}}_{W^{m-4,2}}^{2} + \norm{v_{t}}_{W^{m-4,2}}^{2}d\norm{u^{n}_{t}}_{W^{m-2,2}}^{2}\\ & \qquad \qquad
    + d\norm{u^{n}_{t}}_{W^{m-2,2}}^{2}d\norm{v_{t}}_{W^{m-4,2}}^{2} \\
    &=\left[\norm{u^{n}_{t}}_{W^{m-2,2}}^{2}(J_{1}+J_{2}+J_{3}) + \norm{v_{t}}_{W^{m-4,2}}^{2}(I_{1}+I_{2}+I_{3}) + K\right] dt \\
    & \qquad \qquad + \left[\norm{u^{n}_{t}}_{W^{m-2,2}}^{2} J_{4} + \norm{v_{t}}_{W^{m-4,2}}^{2} I_{4} \right]dW_{t}
\end{align*}
where $K=4\infsum{i}\inner{\proj\Gi u^{n}_{t}}{u^{n}_{t}}_{W^{m-2,2}}\inner{\proj\Gi v_{t}}{v_{t}}_{W^{m-4,2}}$. We now turn to estimating all of the terms. We shall frequently use the nonlinearity estimate (\ref{nonlinear term control}) and the noise control of Propositions \ref{transport and stretching 2} and \ref{transport only 2} without explicit reference, also absorbing the dependence on $(\xi_i)$ into the generic constant. Firstly,
\begin{align}
   \nonumber \left|\norm{v_{t}}_{W^{m-4,2}}^{2}(I_{1}+I_{3})\right| &\lesssim \norm{v_{t}}_{W^{m-4,2}}^{2} \left( \norm{u^{n}_{t}}_{W^{1,\infty}} \norm{u^{n}_{t}}_{W^{m-2,2}}^{2}
    + \norm{u^{n}_{t}}_{W^{m-2,2}}^{2} \right) \\
    & \lesssim \norm{v_{t}}_{W^{m-4,2}}^{2}\norm{u^{n}_{t}}_{W^{m-2,2}}^{2} \left( \norm{u^{n}_{t}}_{W^{1,\infty}} + 
    1\right). \label{a la first}
\end{align}
Secondly,
\begin{align}
    \norm{u^{n}_{t}}_{W^{m-2,2}}^{2} |J_{3}| &\lesssim  \norm{u^{n}_{t}}_{W^{m-2,2}}^{2} \norm{v_{t}}_{W^{m-4,2}}^{2}. \label{a la second}
\end{align}
The contribution of the nonlinear term in $v$ is more delicate, although this is handled in [\cite{glatt2012local}] (7.19), (7.23). In particular,
\begin{equation} \label{a la third} \norm{u^{n}_{t}}_{W^{m-2,2}}^{2} |J_{1}| \lesssim \norm{u^{n}_{t}}_{W^{m-2,2}}^{2}\norm{v_t}_{W^{m-4,2}}^2\left(\norm{u^n}_{W^{m-3,2}} + \norm{u^j}_{W^{m-3,2}} \right). \end{equation}

The $|K|$ term enjoys our noise estimates, with
\begin{align}\label{Kbound}
   |K| &\lesssim \infsum{i}\left|\inner{\proj\Gi u^{n}_{t}}{u^{n}_{t}}_{W^{m-2,2}} \right| \left| \inner{\proj\Gi v_{t}}{v_{t}}_{W^{m-4,2}}\right| 
    \lesssim \norm{u^{n}_{t}}_{W^{m-2,2}}^{2} \norm{v_{t}}_{W^{m-4,2}}^{2}.
\end{align}
We move towards the viscous terms. With the observation (\ref{Stokes is non neg}),
\begin{align} \label{viscosity1}
    \norm{v_{t}}_{W^{m-4,2}}^{2}I_{2} = -2\nu_{n} \norm{v_{t}}_{W^{m-4,2}}^{2}\inner{Au^{n}_{t}}{u^{n}_{t}}_{W^{m-2,2}} \leq 
    0.
\end{align}
Employing the same observation again, 
\begin{align}
   \nonumber \norm{u^{n}_{t}}_{W^{m-2,2}}^{2}J_{2} &= -2\nu_{n}\norm{u^{n}_{t}}_{W^{m-2,2}}^{2}
    \inner{Au^{n}_{t}}{v_{t}}_{W^{m-4,2}} + 2\nu_{j}\norm{u^{n}_{t}}_{W^{m-2,2}}^{2}
    \inner{Au^{j}_{t}}{v_{t}}_{W^{m-4,2}}\\ \nonumber
    &= -2\nu_{j}\norm{u^{n}_{t}}_{W^{m-2,2}}^{2}
    \inner{Av_{t}}{v_{t}}_{W^{m-4,2}} + 2\left(\nu_{j}-\nu_{n}\right)\norm{u^{n}_{t}}_{W^{m-2,2}}^{2}
    \inner{Au^{n}_{t}}{v_{t}}_{W^{m-4,2}}\\ \nonumber
    &\lesssim  \left(\nu_{j}-\nu_{n}\right)\norm{u^{n}_{t}}_{W^{m-2,2}}^{2}
    \inner{Au^{n}_{t}}{v_{t}}_{W^{m-4,2}}\\
    &\lesssim \nu_{j}\norm{u^{n}_{t}}_{W^{m-2,2}}^{2}\left(\norm{u^{n}_{t}}_{W^{m-3,2}}^{2} + \norm{u^{j}_{t}}_{W^{m-3,2}}^{2}\right). \label{viscosity2}
\end{align}
similarly to (\ref{hokey stokey bound}). It is now crucial that (\ref{viscosity2}) will be going to
zero, forced by the decay of $\nu_{j}$ and using the uniform bound (\ref{uniform in viscosity bound}). All that is left to estimate now are the stochastic integrals, which we control similarly to (\ref{BDG type control}):
\begin{align} \nonumber
    &\E \left[ \sup_{r\in [0,t]}\left|\int_{0}^{r} \norm{u^{n}_{s}}_{W^{m-2,2}}^{2}J_{4}d\mathcal{W}_{s} \right| \right] \\ \nonumber & \qquad \qquad \leq c 
    \E \left[ \sup_{r\in [0,t]}\left| \infsum{i} \int_{0}^{r}  \norm{u^{n}_{s}}_{W^{m-2,2}}^{2}  
    \inner{\proj\Gi (u^{n}_{t}-u^{j}_{t})}{u^{n}_{t}-u^{j}_{t}}_{W^{m-4,2}} dW_{s}^{i} \right| \right] \\ \nonumber
    & \qquad \qquad \leq c\E \left[ \left( \int_{0}^{t} \norm{u^{n}_{s}}_{W^{m-2,2}}^{4}  \infsum{i} 
    \inner{\proj\Gi (u^{n}_{t}-u^{j}_{t})}{u^{n}_{t}-u^{j}_{t}}_{W^{m-4,2}}^2 ds \right)^{\frac{1}{2}} \right]\\ \nonumber
    & \qquad \qquad \leq  c\E \left[ \left( \int_{0}^{t} \norm{u^{n}_{s}}_{W^{m-2,2}}^{4}  
    \norm{v_{s}}_{W^{m-4,2}}^{4} ds \right)^{\frac{1}{2}} \right]\\ \nonumber
    & \qquad \qquad \leq c\E \left[ \left( \sup_{r\in[0,t]}
    \norm{u^{n}_{r}}_{W^{m-2,2}}^{2}\norm{v_{r}}_{W^{m-4,2}}^{2} \right)^{\frac{1}{2}}
    \left( \int_{0}^{t} \norm{u^{n}_{s}}_{W^{m-2,2}}^{2}\norm{v_{s}}_{W^{m-4,2}}^{2} ds
    \right)^{\frac{1}{2}} \right]\\
    & \qquad \qquad \leq \frac{1}{4}\E \left[ \sup_{r\in[0,t]}
    \norm{u^{n}_{r}}_{W^{m-2,2}}^{2}\norm{v_{r}}_{W^{m-4,2}}^{2}\right] +
    c \E\left(\int_{0}^{t} 
    \norm{u^{n}_{s}}_{W^{m-2,2}}^{2}\norm{v_{s}}_{W^{m-4,2}}^{2}ds\right). \label{stochest1}
\end{align}
We have used the Burkholder-Davis-Gundy Inequality and in the last line Young's inequality, so that we can subtract the 
supremum term from the left hand side in the final step when we collect all of the
estimates and apply the Gr\"{o}nwall lemma. Identically, we have that

\begin{align} \nonumber 
    &\E \left[ \sup_{r\in [0,t]}\left|\int_{0}^{s} \norm{{v}_{s}}_{W^{m-4,2}}^{2}I_{4}d\mathcal{W}_{s} \right| \right] \\ & \qquad \qquad \leq
    \frac{1}{4}\E \left[ \sup_{s\in[0,t]}
    \norm{u^{n}_{s}}_{W^{m-2,2}}^{2}\norm{v_{s}}_{W^{m-4,2}}^{2}\right] +
    c \E \left(\int_{0}^{t} 
    \norm{u^{n}_{s}}_{W^{m-2,2}}^{2}\norm{v_{s}}_{W^{m-4,2}}^{2}ds\right). \label{stochest2}
\end{align}
Collecting all of the estimates (\ref{a la first})-(\ref{stochest2}), as well as using the uniform control granted by $\alpha$, we can finally write down that
\begin{align*}
    \E \left[ \sup_{r\in[0,t]} \norm{u^{n}_{r}}_{W^{m-2,2}}^{2}\norm{v_{r}}_{W^{m-4,2}}^{2}\right] \lesssim
    \nu_j\mathbbm{E}\left(\int_0^t\norm{u^n_s}_{W^{m-2,2}}^{2} ds\right) +
    \mathbbm{E}\left(\int_{0}^{t} 
    \norm{u^{n}_{s}}_{W^{m-2,2}}^{2}\norm{v_{s}}_{W^{m-4,2}}^{2}ds\right)
\end{align*}
noting that $v_0 = 0$. Employing the bound (\ref{uniform in viscosity bound}) in the viscous term, an application of the standard Gr\"onwall lemma yields that
\begin{align} \nonumber
    \E \left[ \sup_{r\in[0,T]} \norm{u^{n}_{r}}_{W^{m-2,2}}^{2}\norm{v_{r}}_{W^{m-4,2}}^{2}\right] \lesssim
    \nu_{j}
\end{align}
from which we conclude (\ref{ultimate yield}) and ultimately the result.

\end{proof}

\subsection{Weak Equicontinuity} \label{subs weak equi}

\begin{lemma} \label{weak equi lemma}
     Let $\theta$ be a stopping time and $(\delta_l)$ a sequence of stopping times which converge to $0$ $\mathbbm{P}-a.s.$. Then
     $$\lim_{l \rightarrow \infty}\sup_{n\in\N}\mathbbm{E}\left[\sup_{r \in [0 ,(\theta + \delta_l)\wedge \tau^M_n]}\norm{u^n_r}_{W^{m-3,2}}^2 - \sup_{r \in [0 , \theta\wedge \tau^M_n]}\norm{u^n_r}_{W^{m-3,2}}^2 \right] =0. $$

\end{lemma}

\begin{proof}


    Using the observation (75) from [\cite{goodair2024weak}] we simplify the task at hand to
    \begin{align} \label{simplified task}
        \lim_{l \rightarrow \infty}\sup_{n\in\N}\mathbbm{E}\left(\sup_{r \in [0,\delta_l]}\norm{u^n_{(\theta + r) \wedge \tau^{M}_n]}}_{W^{m-3,2}}^2 - \norm{u^n_{\theta \wedge \tau^{M}_n]}}_{W^{m-3,2}}^2 \right) =0.
    \end{align}
    To obtain the above estimate we apply the It\^{o} Formula in $W^{m-3,2}_{\sigma}$ up until the stopping time $\theta \wedge \tau^{M}_n$ and then $(\theta + r) \wedge \tau^M_n$ for some $r \geq 0$, then subtract the two to obtain that
    \begin{align*}
        &\norm{u^{n}_{(\theta + r) \wedge \tau^{M}_n}}^{2}_{W^{m-3,2}} -\norm{u^{n}_{\theta  \wedge \tau^{M}_n}}^{2}_{W^{m-3,2}}
        + 2\nu_{n}\int_{\theta \wedge \tau^{M}_n}^{(\theta + r) \wedge \tau^{M}_n} \inner{Au^{n}_{t}}{u^{n}_{t}}_{W^{m-3,2}}dt\\ & \qquad  = -2\int_{\theta \wedge \tau^{M}_n}^{(\theta + r) \wedge \tau^{M}_n} \inner{\proj \nl{u^{n}_{t}}{u^{n}_{t}}} {u^{n}_{t}}_{W^{m-3,2}}dt + 2\infsum{i} \int_{\theta \wedge \tau^{M}_n}^{(\theta + r) \wedge \tau^{M}_n}  \inner{\proj\Gi u^{n}_{t}}{u^{n}_{t}}_{W^{m-3,2}} dW_{t}^{i}.\\ & \qquad  + 
        \int_{\theta \wedge \tau^{M}_n}^{(\theta + r) \wedge \tau^{M}_n} \infsum{i}\left(\inner{\left(\proj\Gi\right)^{2} u^{n}_{t}}{u^{n}_{t}}_{W^{m-3,2}} + \norm{\proj \Gi u^{n}_{t}}_{W^{m-3,2}}^2\right)dt.
    \end{align*}
    We shall again use (\ref{Stokes is non neg}) to drop the viscous term. As with our previous estimates we shall bound by the absolute value, take the supremum and expectation followed by the Burkholder-Davis-Gundy Inequality. Estimating the nonlinear term by (\ref{nonlinear term control}), and the noise with Propositions \ref{transport and stretching 2} and \ref{transport only 2}, we arrive at 
    \begin{multline*}
        \E \left[\sup_{r \in [0,\delta_l]}\norm{u^n_{(\theta + r) \wedge \tau^{M}_n]}}_{W^{m-3,2}}^2 - \norm{u^n_{\theta \wedge \tau^{M}_n]}}_{W^{m-3,2}}^2 \right] \lesssim \\ \E \left[ \int_{\theta \wedge \tau^{M}_n}^{(\theta + \delta_l) \wedge \tau^{M}_n} \left(\norm{u^{n}_{t}}_{W^{1,\infty}} + 1\right) \norm{u^{n}_{t}}_{W^{m-3,2}}^2 dt \right] +
        \E \left[ \int_{\theta \wedge \tau^{M}_n}^{(\theta + \delta_l) \wedge \tau^{M}_n} \norm{u^{n}_{t}}_{W^{m-3,2}}^{4} dt \right]^{\frac{1}{2}}.
    \end{multline*} 
    Here we use the control granted by $\tau^M_n$ such that both integrands are bounded by a constant, therefore
    \begin{align*}
        \E \left[\sup_{r \in [0,\delta_l]}\norm{u^n_{(\theta + r) \wedge \tau^{M}_n}}_{W^{m-3,2}}^2 - \norm{u^n_{\theta \wedge \tau^{M}_n}}_{W^{m-3,2}}^2 \right] \lesssim \E\left(\delta_{l}\right) + \E\left(\delta_l^{\frac{1}{2}}\right)
    \end{align*}
    from which the Monotone Convergence Theorem with $\delta_l \rightarrow 0$ gives (\ref{simplified task}) hence the result.
    \end{proof}

\subsection{Existence of Local Smooth Solutions} \label{subs smooth}

We at first deduce the existence of our candidate local strong solution.

\begin{lemma} \label{candidate solution}
    There exists a stopping time $\tau^{M}_{\infty}$, a process $u:\Omega \mapsto C\left([0,\tau^{M}_{\infty}];W^{m-3,2}_{\sigma}\right)$ whereby $\sup_{r \in [0, \cdot \wedge \tau^{M}_{\infty}]}\norm{u_r}_{W^{m-3,2}}^2$ is adapted and $\mathbbm{P}-a.s.$ continuous, and a subsequence indexed by $(n_j)$ such that 
\begin{enumerate}
    \item $\tau^{M}_{\infty} \leq \tau^{M}_{n_j}$ $\mathbbm{P}-a.s.$,
    \item \label{second item} $\lim_{j \rightarrow \infty}\sup_{r \in [0 ,\tau^{M}_{\infty}]}\norm{u_r - u^{n_j}_r}_{W^{m-3,2}}^2 = 0$ $\mathbbm{P}-a.s.$.
\end{enumerate}
Moreover for any $R>0$ we can choose $M$ to be such that the stopping time \begin{equation} \label{another tauR}
        \tau^{R} := T \wedge \inf\left\{s \geq 0: \sup_{r \in [0 ,s \wedge \tau^{M}_{\infty}]}\norm{u_r}_{W^{m-3,2}}^2 \geq R \right\}
    \end{equation}
satisfies $\tau^{R} \leq \tau^{M}_{\infty}$ $\mathbbm{P}-a.s.$. Thus $\tau^{R}$ is simply $T \wedge \inf\left\{s \geq 0: \sup_{r \in [0 ,s ]}\norm{u_r}_{W^{m-3,2}}^2  \geq R \right\}$.
\end{lemma}

\begin{proof}
    This is a direct application of Proposition \ref{amazing cauchy lemma}, for the spaces $X_t := C\left([0,t];W^{m-3,2}_{\sigma} \right)$, where condition (\ref{supposition 1}) is shown in Proposition \ref{cauchy prop} and condition (\ref{supposition 2}) in Lemma \ref{weak equi lemma}. 
\end{proof}

\begin{proposition} \label{smooth solution existence prop}
    Fix any $R > \norm{u_0}_{L^\infty\left(\Omega;W^{m-3,2}\right)}^2$, choose $M$ and define $\tau^R$ as in Lemma \ref{candidate solution}. Then $(u_{\cdot \wedge \tau^R},\tau^R)$ is a local $W^{m-3,2}_{\sigma}-$strong solution of the equation (\ref{general itoo euler proper}), $u_{\cdot}\mathbbm{1}_{\cdot \wedge \tau^R}$ has a progressively measurable version\footnote{By a version, we mean a process $\py_{\cdot}$ such that $\py_{\cdot} = u_{\cdot}\mathbbm{1}_{\cdot \wedge \tau^R}$ $\mathbbm{P} \times \lambda-a.s.$, for $\lambda$ the Lebesgue measure, over the product space $\Omega \times [0,T]$.} in $W^{m-1,2}_{\sigma}$ and belongs $\mathbbm{P}-a.s.$ to $L^\infty\left([0,T];W^{m-1,2}_{\sigma}\right)$.
\end{proposition}

\begin{proof}
    We first note that, as in the remark following Proposition \ref{amazing cauchy lemma}, $\tau^R$ is $\mathbbm{P}-a.s.$ positive. The regularity in $W^{m-3,2}_{\sigma}$ follows from item \ref{second item} as the $\mathbbm{P}-a.s.$ limit of adapted and continuous processes. For the regularity in $W^{m-1,2}_{\sigma}$ we use that $ \tau^R \leq \tau^{M}_{n_j}$ $\mathbbm{P}-a.s.$, so the uniform estimates (\ref{uniform in viscosity bound}) hold for the subsequence up until $\tau^R$. 
Relabelling the subsequence to $(u^n)$ for simplicity, then $\left(u^n_{\cdot}\mathbbm{1}_{\cdot \wedge \tau^R}\right)$ is uniformly bounded in $L^2\left(\Omega;L^\infty\left([0,T];W^{m-1,2}_{\sigma}\right)\right)$, such that we can extract a further subsequence which is weak* convergent by the Banach-Alaoglu Theorem, identifying $L^2\left(\Omega;L^\infty\left([0,T];W^{m-1,2}_{\sigma}\right)\right)$ with the dual of $L^2\left(\Omega;L^1\left([0,T];W^{m-1,2}_{\sigma}\right)\right)$. We know that $\left(u^n_{\cdot}\mathbbm{1}_{\cdot \wedge \tau^R}\right)$ converges to $u_{\cdot}\mathbbm{1}_{\cdot \wedge \tau^R}$ in $L^2\left(\Omega;L^\infty\left([0,T];W^{m-3,2}_{\sigma}\right)\right)$ by item \ref{second item} and the Dominated Convergence Theorem with the uniform bounds given by $\tau^R$, hence both convergences hold in the weak* topology of $L^2\left(\Omega;L^\infty\left([0,T];W^{m-3,2}_{\sigma}\right)\right)$ so by uniqueness of limits in this topology then $u_{\cdot}\mathbbm{1}_{\cdot \wedge \tau^R}$ is the weak* limit in $L^2\left(\Omega;L^\infty\left([0,T];W^{m-1,2}_{\sigma}\right)\right)$. Consequently $u_{\cdot}\mathbbm{1}_{\cdot \wedge \tau^R}$ belongs $\mathbbm{P}-a.s.$ to $L^\infty\left([0,T];W^{m-1,2}_{\sigma}\right)$. For the progressive measurability, we observe that the convergence holds in $L^2\left(\Omega \times [0,t];W^{m-1,2}_{\sigma}\right)$ for every $0 \leq t \leq T$. As the approximating sequence is progressively measurable we can equip $\Omega \times [0,t]$ with the $\mathcal{F}_t \times \mathcal{B}\left([0,t]\right)$ sigma-algebra, hence the limit is measurable with respect to this sigma-algebra which justifies the progressive measurability of the version obtained from the limit.\\

 With this regularity established it now only remains to show that the limiting pair satisfies the identity (\ref{local identity}). For this we refer to [\cite{goodair2025navier}] Proposition 3.5, where the inviscid limit of Navier-Stokes with transport-stretching noise is shown to satisfy the stochastic Euler equation weakly under a much weaker topology of convergence. The transport noise case follows in exactly the same way given the estimates of Proposition \ref{transport only 2}, so we omit further details and conclude the proof. 

\end{proof}

\subsection{Optimising Smoothness of the Initial Condition} \label{subs initial condition relax}

We improve the results of the previous subsection by showing that we can construct a solution with regularity matching the smoothness of the initial condition.

\begin{proposition} \label{prop relaxing smooth}
Fix $m' \geq 3$, let $u_0 \in L^\infty\left( \Omega ; W^{m',2}_{\sigma} \right)$ be $\mathcal{F}_0-$measurable and each $\xi_i \in L^2_{\sigma} \cap  W^{m'+6,\infty}$ such that $\sum_{i=1}^\infty \norm{\xi_i}_{W^{m'+5,\infty}}^2 < \infty$. Fix any $R' > \norm{u_0}_{L^\infty\left(\Omega;W^{m',2}\right)}^2$. There exists a pair $(u_{\cdot \wedge \tau^{R'}}, \tau^{R'})$ which is a local $W^{m',2}_{\sigma}-$strong solution of the equation (\ref{general itoo euler proper}), where
$$\tau^{R'} = T \wedge \inf\left\{s \geq 0: \sup_{r \in [0 ,s ]}\norm{u_r}_{W^{m',2}}^2  \geq R' \right\}.$$

    
\end{proposition}

\begin{proof}
To construct the desired local strong solution, we consider an approximating sequence of local strong solutions to (\ref{general itoo euler proper}) for an initial condition regularised by the projection operator $\mathcal{P}_n$ specified in Subsection \ref{subs functional anal}. For every $n\in \N$ we have that $\mathcal{P}_nu_0 \in L^\infty\left(\Omega;W^{m'+3,2}_{\sigma} \right)$. Therefore if we fix any $M' > 1$, for every $n' \in \N$, by Proposition \ref{smooth solution existence prop} there exists a local $W^{m',2}_{\sigma}-$strong solution $(u^{n'}_{\cdot \wedge \tau^{M'}_{n'}}, \tau^{M'}_{n'})$ of the equation (\ref{general itoo euler proper}) for the initial condition $u^{n'}_0 = \mathcal{P}_{n'}u_0$, whereby
$$\tau^{M'}_{n'} = T \wedge \inf\left\{s \geq 0: \sup_{r \in [0 ,s ]}\norm{u^{n'}_r}_{W^{m',2}}^2  \geq M' + \norm{u_0}_{W^{m',2}}^2 \right\}.$$
 Proposition \ref{smooth solution existence prop} is invoked by taking $m$ of Proposition \ref{navier stokes existence prop} to be $m'+3$, and choosing $R > M' + \norm{u_0}_{L^\infty\left(\Omega;W^{m',2}\right)}^2$. Furthermore Proposition \ref{smooth solution existence prop} implies that $u^{n'}_{\cdot}\mathbbm{1}_{\cdot \leq \tau^{M'}_{n'}}$ has a progressively measurable version in $W^{m'+2,2}_{\sigma}$ and belongs $\mathbbm{P}-a.s.$ to $L^\infty\left([0,T]; W^{m'+2,2}_{\sigma} \right)$. These solutions with regularised initial condition will now play the role of the approximating sequence of Navier-Stokes equations used in proving Proposition \ref{smooth solution existence prop}. Indeed we claim two properties:
\begin{enumerate}
    \item  We have that
    $$\lim_{j' \rightarrow \infty}\sup_{n' \geq j'} \mathbbm{E} \left( \sup_{r \in [0,\tau^{M'}_{n'} \wedge \tau^{M'}_{j'}]}\norm{u^{n'}_r - u^{j'}_r}_{W^{m',2}}^2 \right) = 0. $$
    
    \item   Let $\theta$ be a stopping time and $(\delta_l)$ a sequence of stopping times which converge to $0$ $\mathbbm{P}-a.s.$. Then
    $$\lim_{l \rightarrow \infty}\sup_{n'\in\N}\mathbbm{E}\left[\sup_{r \in [0 ,(\theta + \delta_l)\wedge \tau^{M'}_{n'}]}\norm{u^{n'}_r}_{W^{m',2}}^2 - \sup_{r \in [0 , \theta\wedge \tau^{M'}_{n'}]}\norm{u^{n'}_r}_{W^{m',2}}^2 \right] =0. $$
\end{enumerate}
A verification of these properties is near identical to their counterparts of Proposition \ref{cauchy prop} and Lemma \ref{weak equi lemma}. We must note that the additional $W^{m' + 2,2}_{\sigma}$ regularity ensures that the solutions satisfy the identity of (\ref{local identity}) in $W^{m',2}_{\sigma}$ so that we can apply the It\^{o} Formula in this space. For the Cauchy property, instead of the viscosity appearing on the right hand side of the estimates we have $\norm{\mathcal{P}_{n'}u_0 - \mathcal{P}_{j'}u_0}_{W^{m',2}}^2$ which similarly approaches zero. The weak equicontinuity proof is unchanged aside from the absence of the viscous term. From here, the proof follows exactly as in Subsection \ref{subs smooth}.


\end{proof}

\subsection{Uniqueness of Local Strong Solutions} \label{subs unique}

\begin{proposition} \label{Prop unique}
    Fix $3 \leq m \in \N$ and let $u_0: \Omega \rightarrow W^{m,2}_{\sigma}$ be $\mathcal{F}_0-$ measurable. Suppose that $(u,\tau)$, $(v, \theta)$ are two local $W^{m,2}_{\sigma}-$strong solutions of the equation (\ref{general itoo euler proper}). Then \begin{equation} \nonumber\mathbbm{P}\left(\left\{\omega \in \Omega: u_{t}(\omega) =  v_{t}(\omega)  \quad \forall t \in [0, T \wedge \tau(\omega) \wedge \theta(\omega)] \right\} \right) = 1. \end{equation}
\end{proposition}

\begin{proof}
   Uniqueness of weak solutions to the 2D stochastic Navier-Stokes equation, for a general noise including transport and transport-stretching, was shown in [\cite{goodair2023zero}] Proposition 3.10. Given the additional regularity on solutions here our task is simpler and contained in the proof from [\cite{goodair2023zero}], hence we omit the details. 
\end{proof}

\subsection{Maximality and Blow-Up} \label{subs maximality}

We now demonstrate that the unique $W^{m,2}_{\sigma}-$strong solution extends to a maximal $W^{m,2}_{\sigma}-$strong solution, and prove our blow-up criterion.




\begin{proposition} \label{penultimate prop}
  Fix $3 \leq m \in \N$, let $u_0 \in L^\infty\left( \Omega ; W^{m,2}_{\sigma} \right)$ be $\mathcal{F}_0-$measurable and each $\xi_i \in L^2_{\sigma} \cap  W^{m+6,\infty}$ such that $\sum_{i=1}^\infty \norm{\xi_i}_{W^{m+5,\infty}}^2 < \infty$.  There exists a unique maximal $W^{m,2}_{\sigma}-$strong solution $(u,\Theta)$ of the equation (\ref{general itoo euler proper}) with the properties that:
\begin{enumerate}
    \item \label{a first item} At $\mathbbm{P}-a.e.$ $\omega$ for which $\Theta(\omega)<T$, we have that \begin{equation} \nonumber \sup_{r \in [0 ,\Theta(\omega))}\norm{u_r(\omega)}_{W^{m,2}}^2  = \infty.\end{equation}
    \item \label{a second item} If $\tau$ is a stopping time such that for $\mathbbm{P}-a.e.$ $\omega$, $\tau(\omega) \in (0,T]$ and $$\sup_{r \in [0 ,\tau(\omega))}\norm{u_r(\omega)}_{W^{m,2}}^2 < \infty,$$
     then $(u_{\cdot \wedge \tau}, \tau)$ is a local $W^{m,2}_{\sigma}-$strong solution of the equation (\ref{general itoo euler proper}). 
\end{enumerate}
\end{proposition}

\begin{proof}
    The existence of a unique maximal $W^{m,2}_{\sigma}-$strong solution $(u,\Theta)$ follows directly from the existence (Proposition \ref{prop relaxing smooth}) and uniqueness (Proposition \ref{Prop unique}) of local $W^{m,2}_{\sigma}-$strong solutions, for which we refer to [\cite{goodair2024stochastic}] Theorem 3.3. We note that adaptedness here is truly of the solution and not a version of it, so the progressive measurability stressed by the `regular' solution of [\cite{goodair2024stochastic}] Theorem 3.3 is satisfied. Towards the blow-up criterion, we recall [\cite{goodair2024stochastic}] Corollary 3.1 that if $\gamma$ were any stopping time such that $(u_{\cdot \wedge \gamma}, \gamma)$ is a local $W^{m,2}_{\sigma}-$strong solution then $\gamma \leq \Theta$ $\mathbbm{P}-a.s.$. By uniqueness then for any $R' > \norm{u_0}_{L^\infty\left(\Omega;W^{m,2}\right)}^2$, the local $W^{m,2}_{\sigma}-$strong solution $(u_{\cdot \wedge \tau^{R'}}, \tau^{R'})$ constructed in Proposition \ref{prop relaxing smooth} must be the stopped process and first hitting time of the unique maximal solution. Therefore, for every $R>\norm{u_0}_{L^\infty\left(\Omega;W^{m,2}\right)}^2$ the stopping time $\tau^R$ given by 
    \begin{equation} \label{given taur} \tau^{R} = T \wedge \inf\left\{s \geq 0: \sup_{r \in [0,s]}\norm{u_r}_{W^{m,2}}^2  \geq R \right\}
    \end{equation}
    satisfies that $\tau^R \leq \Theta$ $\mathbbm{P}-a.s.$. Let us define the set
    $$ \mathscr{A} = \left\{\omega \in \Omega: \Theta(\omega) < T, \sup_{r \in [0,\Theta(\omega))}\norm{u_r(\omega)}_{W^{m,2}}^2  < \infty \right\}.$$
    To prove item \ref{a first item} it is sufficient to show that $\mathbbm{P}(\mathscr{A}) = 0$. Observing that
    $$\mathscr{A} = \bigcup_k \mathscr{A}_k, \qquad \mathscr{A}_k = \left\{\omega \in \Omega: \Theta(\omega) < T, \sup_{r \in [0,\Theta(\omega))}\norm{u_r(\omega)}_{W^{m,2}}^2  \leq k \right\} $$
    then it is instead sufficient to show that $\mathbbm{P}(\mathscr{A}_k) = 0$ for every $k$.  $\mathbbm{P}-a.s.$ in $\mathscr{A}_k$, as $\tau^{k+1} \leq \Theta < T$ certainly $\tau^{k+1} < T$ so $\tau^{k+1} = \inf\left\{s \geq 0: \sup_{r \in [0,s]}\norm{u_r}_{W^{m,2}}^2  \geq k+1 \right\}$. As $\tau^{k+1} \leq \Theta$ then $\sup_{r \in [0 ,\Theta)}\norm{u_s}_{W^{m,2}}^2  \geq k+1$ which would contradict the definition of $\mathscr{A}_k$, hence $\mathbbm{P}(\mathscr{A}_k) = 0$ as required. In the direction of item \ref{a second item} let us take any stopping time $\tau \in (0,T]$ such that 
    $$\sup_{r \in [0 ,\tau)}\norm{u_r}_{W^{m,2}}^2 < \infty$$
    $\mathbbm{P}-a.s.$. Firstly we claim that this implies
$$\sup_{r \in [0 ,\tau]}\norm{u_r}_{W^{m,2}}^2 < \infty$$
$\mathbbm{P}-a.s.$ Indeed let us take any $\omega$ from the full probability set for which the supremum is finite and $u_{\cdot \wedge \tau^R}$ is continuous for all $R$. Now $$\sup_{r \in [0 ,\tau(\omega))}\norm{u_r(\omega)}_{W^{m,2}}^2 \leq M$$ for some $M$ (dependent on $\omega$). However as $u_{\cdot \wedge \tau^{M+1}(\omega)}(\omega)$ is continuous then $u_{\cdot}(\omega)$ cannot explode at $\tau(\omega)$, justifying the claim. For any $R > \norm{u_0}_{L^\infty\left(\Omega;W^{m,2}\right)}^2$ we have that $(u_{\cdot \wedge \tau \wedge \tau^R}, \tau \wedge \tau^R)$ is a local $W^{m,2}_{\sigma}-$strong solution. In addition for $\mathbbm{P}-a.e.$ $\omega$ there exists an $R$, dependent on $\omega$, such that $\tau^R(\omega) \geq \tau(\omega)$. Therefore, $u_{\cdot \wedge \tau}$ is given as the limit of $(u_{\cdot \wedge \tau \wedge \tau^R})$ $\mathbbm{P}-a.s.$ in $C\left([0,T];W^{m,2}_{\sigma} \right)$, hence retains the adaptedness and continuity of each $u_{\cdot \wedge \tau \wedge \tau^R}$, whilst also solving the identity (\ref{local identity}) as required.

\end{proof}

\begin{lemma} \label{key blowup}
    Let $(u,\Theta)$ be the unique maximal $W^{m,2}_{\sigma}-$strong solution specified in Proposition \ref{penultimate prop}. At $\mathbbm{P}-a.e.$ $\omega$ for which $\Theta(\omega)<T$, we have that \begin{equation} \nonumber \int_{0}^{\Theta(\omega)}\norm{u_s(\omega)}_{W^{1,\infty}}ds  = \infty.\end{equation}
\end{lemma}

\begin{proof}
    We define, for any given $M > 0$, the stopping time
    $$\gamma^M := \Theta \wedge \inf\left\{s \geq 0: \int_{0}^{s}\norm{u_r}_{W^{1,\infty}}ds  \geq M \right\}.$$
    The proof relies on showing that for any $R>\norm{u_0}_{L^\infty\left(\Omega;W^{m,2}\right)}^2$ and the stopping time $\tau^R$ defined in (\ref{given taur}), we have that 
 \begin{equation}  \label{the implier}
   \mathbbm{E}\left(\sup_{r \in [0 , \gamma^M \wedge \tau^R]}\norm{u_r}_{W^{m,2}}^2\right) \leq C_M 
 \end{equation}
where $C_M$ is independent of $R$. Let us first consider how (\ref{the implier}) will imply the result, and the prove (\ref{the implier}) later. We observe that $(\tau^{R})$ are $\mathbbm{P}-a.s.$ monotone increasing to $\Theta$; indeed $(\tau^R) \leq \Theta$ and are clearly $\mathbbm{P}-a.s.$ monotone increasing so admit a limit $\beta$. Thus $\beta \leq \Theta$, but if $\beta < \Theta$ on a set of positive probability then on that set, by the definition of the maximal time, there must be a stopping time $\kappa$ such that $\beta < \kappa \leq \Theta$ and $(u_{\cdot \wedge \kappa},\kappa)$ is a local $W^{m,2}_{\sigma}-$strong solution. However as $\beta < \kappa$ then $u_{\cdot \wedge \kappa}$ cannot be continuous in $W^{m,2}_{\sigma}$ which provides a contradiction hence the observation. Therefore we may apply the Monotone Convergence Theorem to (\ref{the implier}) to deduce that 
 \begin{equation}  \label{the implier 2}
   \mathbbm{E}\left(\sup_{r \in [0 , \gamma^M \wedge \Theta)}\norm{u_r}_{W^{m,2}}^2\right) \leq C_M. 
 \end{equation}
Of course $\gamma^M \wedge \Theta = \gamma^M$, so from (\ref{the implier 2}) we deduce that \begin{equation} \label{to contradict} \sup_{r \in [0 , \gamma^M)}\norm{u_r}_{W^{m,2}}^2 < \infty\end{equation}
$\mathbbm{P}-a.s.$. We now claim that \begin{equation} \label{the claim}\gamma^M  = T \wedge \inf\left\{s \geq 0: \int_{0}^{s}\norm{u_r}_{W^{1,\infty}}ds  \geq M \right\}.\end{equation}
On $\Omega$ for which $\Theta = T$ then this is trivial, and for $\Theta < T$ then by Proposition \ref{penultimate prop} item \ref{a first item} we must have that
$$ \sup_{r \in [0 , \Theta)}\norm{u_r}_{W^{m,2}}^2 = \infty.$$
If $\gamma^M = \Theta$ then this would contradict (\ref{to contradict}), so on $\Omega$ for which $\Theta < T$ we must have that
$$\gamma^M = \inf\left\{s \geq 0: \int_{0}^{s}\norm{u_r}_{W^{1,\infty}}ds  \geq M \right\} = T \wedge \inf\left\{s \geq 0: \int_{0}^{s}\norm{u_r}_{W^{1,\infty}}ds  \geq M \right\} $$
as $\gamma^M < \Theta < T$. Therefore, the claim (\ref{the claim}) is justified. Now (\ref{to contradict}) together with Proposition \ref{penultimate prop} item \ref{a second item} implies that $(u_{\cdot \wedge \gamma^M}, \gamma^M)$ is a local $W^{m,2}_{\sigma}-$strong solution. However due to the representation (\ref{the claim}), then we can prove the desired blow-up criterion exactly as we proved Proposition \ref{penultimate prop} item \ref{a first item}, with $\gamma^M$ replacing $\tau^R$ in (\ref{given taur}).\\

Therefore we can prove the Lemma if we verify (\ref{the implier}). Note that we cannot take energy estimates of $u$ directly in $W^{m,2}_{\sigma}$ as $u$ does not satisfy the identity in this space. Thus we must turn to the approximating sequence used to construct the local $W^{m,2}_{\sigma}-$strong solution up until $\tau^R$, which was the content of Proposition \ref{prop relaxing smooth}. For $(u^{n'}_{\cdot \wedge \tau^{M'}_{n'}},\tau^{M'}_{n'})$ the local $W^{m,2}_{\sigma}-$strong solutions of the equation (\ref{general itoo euler proper}) with initial condition $u^{n'}_0 = \mathcal{P}_{n'}u_0$, recalling also the details from Subsection \ref{subs smooth}, then $M'$ can be chosen sufficiently large such that for a subsequence relabelled from $n'$ to $n$, $\tau^R \leq \tau^{M'}_{n}$ and $$\lim_{n \rightarrow \infty}\sup_{r \in [0,\tau^R]}\norm{u_r - u^{n}_r}_{W^{m,2}}^2 = 0 $$
holds $\mathbbm{P}-a.s.$. Of course the convergence holds up until $\gamma^M \wedge \tau^R$, and we rewrite this as
\begin{equation}\label{abrewrite} \lim_{n \rightarrow \infty}\sup_{r \in [0,T]}\norm{u_{r \wedge \tau^R \wedge \gamma^M} - u^{n}_{r \wedge \tau^R \wedge \gamma^M}}_{W^{m,2}}^2 = 0\end{equation}
by the continuity of the processes. Unfortunately we will not be able to deduce a uniform control on the approximating sequence working only up until $\gamma^M$, so instead need to introduce the sequence of stopping times $(\gamma^{M+1}_n)$ defined by 
$$\gamma^{M+1}_n  = \tau^R \wedge \inf\left\{s \geq 0: \int_{0}^{s}\norm{u^n_r}_{W^{1,\infty}}ds  \geq M + 1 \right\}.$$
Note that use of $\tau^R$ in this definition is only to ensure that the stopping time is well-defined, as $u^n$ certainly exists up until $\tau^R$. In order to justify using these times, we first need to argue that 
\begin{equation} \label{an argument}\lim_{n \rightarrow \infty}\sup_{r \in [0,T]}\norm{u_{r \wedge \tau^R \wedge \gamma^M} - u^{n}_{r \wedge \tau^R \wedge \gamma^M \wedge \gamma^{M+1}_n}}_{W^{m,2}}^2 = 0\end{equation}
$\mathbbm{P}-a.s.$. Observe that
\begin{align*}
    &\norm{u_{r \wedge \tau^R \wedge \gamma^M} - u^{n}_{r \wedge \tau^R \wedge \gamma^M \wedge \gamma^{M+1}_n}}_{W^{m,2}}^2\\
    & \qquad \leq 2\norm{u_{r \wedge \tau^R \wedge \gamma^M} - u_{r \wedge \tau^R \wedge \gamma^M \wedge \gamma^{M+1}_n}}_{W^{m,2}}^2 + 2\norm{u_{r \wedge \tau^R \wedge \gamma^M \wedge \gamma^{M+1}_n} - u^{n}_{r \wedge \tau^R \wedge \gamma^M \wedge \gamma^{M+1}_n}}_{W^{m,2}}^2
\end{align*}
and we deal with the two terms individually. The second is straightforward as
\begin{align*}
   \lim_{n \rightarrow \infty}\sup_{r \in [0,T]} \norm{u_{r \wedge \tau^R \wedge \gamma^M \wedge \gamma^{M+1}_n} - u^{n}_{r \wedge \tau^R \wedge \gamma^M \wedge \gamma^{M+1}_n}}_{W^{m,2}}^2 \leq \lim_{n \rightarrow \infty}\sup_{r \in [0,T]}\norm{u_{r \wedge \tau^R \wedge \gamma^M} - u^{n}_{r \wedge \tau^R \wedge \gamma^M}}_{W^{m,2}}^2
\end{align*}
which is zero due to (\ref{abrewrite}). To deal with the first term let us consider $\omega$ in the full probability set such that the convergence of (\ref{abrewrite}) holds and the solutions $(u^n_{\cdot \wedge \tau^R})$, $u_{\cdot \wedge \tau^R}$ are continuous. By the embedding of $W^{m,2}_{\sigma} \xhookrightarrow{} W^{1,\infty}$ then the convergence also holds in $L^1\left([0,T];W^{1,\infty} \right)$ and we can choose an $N$ (dependent on $\omega$) such that for all $n \geq N$,
$$ \int_0^{T}\norm{u_{s \wedge \tau^R \wedge \gamma^M} - u^{n}_{s \wedge \tau^R \wedge \gamma^M}}_{W^{1,\infty}}ds < \frac{1}{2}$$
which implies that
\begin{equation} \label{la particular}\int_0^{\tau^R \wedge \gamma^M}\norm{u_{s} - u^{n}_{s}}_{W^{1,\infty}}ds < \frac{1}{2}.
\end{equation}
We now claim that for any $n \geq N$, $u_{r \wedge \tau^R \wedge \gamma^M} = u_{r \wedge \tau^R \wedge \gamma^M \wedge \gamma^{M+1}_n}$ which would imply the desired convergence. The only way in which this could not be true is if $\tau^R \wedge \gamma^M \wedge \gamma^{M+1}_n < \tau^R \wedge \gamma^M$, which can only occur if the hitting threshold in $\gamma^{M+1}_n$ is reached before $\tau^R$ and $\gamma^M$. In particular we would have to have that
$$\int_0^{\tau^R \wedge \gamma^M}\norm{u^{n}_{s}}_{W^{1,\infty}}ds \geq  M + 1$$ but given that $$\int_0^{\tau^R \wedge \gamma^M}\norm{u_{s}}_{W^{1,\infty}}ds \leq M$$
then this would violate (\ref{la particular}) which yields the claim and hence the convergence (\ref{an argument}). Moreover by the uniform control afforded by $\tau^R$, then we may apply the Dominated Convergence Theorem to deduce that 
\begin{equation} \label{an argument result}\lim_{n \rightarrow \infty}\mathbbm{E}\left(\sup_{r \in [0,T]}\norm{u_{r \wedge \tau^R \wedge \gamma^M} - u^{n}_{r \wedge \tau^R \wedge \gamma^M \wedge \gamma^{M+1}_n}}_{W^{m,2}}^2\right) = 0\end{equation}
Therefore to prove (\ref{the implier}) it is sufficient to show that
\begin{equation}  \label{the implier 3}
  \mathbbm{E}\left(\sup_{r \in [0,T]}\norm{u^{n}_{r \wedge \tau^R \wedge \gamma^M \wedge \gamma^{M+1}_n}}_{W^{m,2}}^2\right) =  \mathbbm{E}\left(\sup_{r \in [0 , \gamma^{M+1}_n \wedge \gamma^M \wedge \tau^R]}\norm{u^{n}_r}_{W^{m,2}}^2\right) \leq C_M. 
 \end{equation}
 The proof of this fact is very close to the proof for (\ref{uniform in viscosity bound}), dropping the viscous term and using that $\norm{u^n_0}_{W^{m,2}_{\sigma}}^2 \lesssim \norm{u_0}_{W^{m,2}_{\sigma}}^2.$ The key difference is the role of $\gamma^{M+1}_n$ compared to $\tau^M_n$. When we reach (\ref{granted}) we cannot simply apply a bound by $C_M$ to $\norm{\check{u}^n_s}_{W^{1,\infty}}$; instead we apply the Stochastic Gr\"{o}nwall Lemma \ref{gronny}, with $\eta_s = \norm{\check{u}^n_s}_{W^{1,\infty}}$. Note that we would have to arrive at (\ref{granted}) for given stopping times $\theta_j < \theta_k$ replacing $0,t$ respectively but this is inconsequential to the arguments, see the original paper [\cite{glatt2009strong}] for example. We content ourselves with this as a proof of (\ref{the implier 3}), and ultimately of the lemma.
\end{proof}

\subsection{Final Steps of Theorem \ref{main existence for euler}} \label{subsection final steps}

We now present the final details in the proof of Theorem \ref{main existence for euler}, for which the hard work is now done and we rely on established machinery. Firstly we must extend the existence and uniqueness of a maximal $W^{m,2}_{\sigma}-$strong solution from Proposition \ref{penultimate prop}, and the blow-up criterion of Lemma \ref{key blowup}, to the case of an unbounded initial condition. A complete argument is given in [\cite{goodair2022existence1}] Subsection 3.7 for an abstract SPDE, where one splices the initial condition into bounded parts as seen in [\cite{glatt2009strong}, \cite{glatt2012local}], and the blow-up characterisation of the maximal time is preserved. With this we justify the core statement of Theorem \ref{main existence for euler} and item \ref{main one}. For item \ref{main three} we use the infinite dimensional It\^{o}-Stratonovich conversion proved in [\cite{goodair2024stochastic}] Theorem 3.4 (Corollary 3.2), taking the spaces
$$V:= W^{3,2}_{\sigma}, \quad H:= W^{2,2}_{\sigma}, \quad U:= W^{1,2}_{\sigma}, \quad X:=L^2_{\sigma}.$$
We note that the required solution regularity of $C\left([0,T];W^{2,2}_{\sigma}\right) \cap L^2\left([0,T];W^{3,2}_{\sigma}\right)$ is implied by our regularity of $C\left([0,T];W^{3,2}_{\sigma}\right)$. With this we conclude the proof of Theorem \ref{main existence for euler}.

\appendix
\section{Appendix} \label{appendix}

We collect useful results from the literature that have been used throughout the paper.

\begin{proposition} \label{Ito formula}
Let $\mathcal{H}_1 \subset \mathcal{H}_2 \subset \mathcal{H}_3$ be a triplet of embedded Hilbert Spaces where $\mathcal{H}_1$ is dense in $\mathcal{H}_2$, with the property that there exists a continuous nondegenerate bilinear form $\inner{\cdot}{\cdot}_{\mathcal{H}_3 \times \mathcal{H}_1}: \mathcal{H}_3 \times \mathcal{H}_1 \rightarrow \R$ such that for $\phi \in \mathcal{H}_2$ and $\psi \in \mathcal{H}_1$, $$\inner{\phi}{\psi}_{\mathcal{H}_3 \times \mathcal{H}_1} = \inner{\phi}{\psi}_{\mathcal{H}_2}.$$ Suppose that for some $T > 0$ and stopping time $\tau$,
\begin{enumerate}
        \item $\sy_0:\Omega \rightarrow \mathcal{H}_2$ is $\mathcal{F}_0-$measurable;
        \item $f:\Omega \times [0,T] \rightarrow \mathcal{H}_3$ is such that for $\mathbbm{P}-a.e.$ $\omega$, $f(\omega) \in L^2([0,T];\mathcal{H}_3)$;
        \item $B:\Omega \times [0,T] \rightarrow \mathscr{L}^2(\mathfrak{U};\mathcal{H}_2)$ is progressively measurable and such that for $\mathbbm{P}-a.e.$ $\omega$, $B(\omega) \in L^2\left([0,T];\mathscr{L}^2(\mathfrak{U};\mathcal{H}_2)\right)$;
        \item  \label{4*} $\sy:\Omega \times [0,T] \rightarrow \mathcal{H}_1$ is such that for $\mathbbm{P}-a.e.$ $\omega$, $\sy_{\cdot}(\omega)\mathbbm{1}_{\cdot \leq \tau(\omega)} \in L^2([0,T];\mathcal{H}_1)$ and $\sy_{\cdot}\mathbbm{1}_{\cdot \leq \tau}$ is progressively measurable in $\mathcal{H}_1$;
        \item \label{item 5 again*} The identity
        \begin{equation} \label{newest identity*}
            \sy_t = \sy_0 + \int_0^{t \wedge \tau}f_sds + \int_0^{t \wedge \tau}B_s d\mathcal{W}_s
        \end{equation}
        holds $\mathbbm{P}-a.s.$ in $\mathcal{H}_3$ for all $t \in [0,T]$.
    \end{enumerate}
The the equality 
  \begin{align} \label{ito big dog*}\norm{\sy_t}^2_{\mathcal{H}_2} = \norm{\sy_0}^2_{\mathcal{H}_2} + \int_0^{t\wedge \tau} \bigg( 2\inner{f_s}{\sy_s}_{\mathcal{H}_3 \times \mathcal{H}_1} + \norm{B_s}^2_{\mathscr{L}^2(\mathfrak{U};\mathcal{H}_2)}\bigg)ds + 2\int_0^{t \wedge \tau}\inner{B_s}{\sy_s}_{\mathcal{H}_2}d\mathcal{W}_s\end{align}
  holds for any $t \in [0,T]$, $\mathbbm{P}-a.s.$ in $\R$. Moreover for $\mathbbm{P}-a.e.$ $\omega$, $\sy_{\cdot}(\omega) \in C([0,T];\mathcal{H}_2)$. 
\end{proposition}

\begin{proof}
    See [\cite{goodair2024stochastic}] Proposition 4.3, a slight extension of [\cite{prevot2007concise}] Lemma 4.2.5.
\end{proof}

\begin{proposition} \label{amazing cauchy lemma}
    Fix $T>0$. For $t\in[0,T]$ let $X_t$ denote a Banach Space with norm $\norm{\cdot}_{X,t}$ such that for all $s > t$, $X_s \xhookrightarrow{}X_t$ and $\norm{\cdot}_{X,t} \leq \norm{\cdot}_{X,s}$. Suppose that $(\sy^n)$ is a sequence of processes $\sy^n:\Omega \mapsto X_T$, $\norm{\sy^n}_{X,\cdot}$ is adapted and $\mathbbm{P}-a.s.$ continuous, $\sy^n \in L^2\left(\Omega;X_T\right)$, and such that $\sup_{n}\norm{\sy^n}_{X,0} \in L^\infty\left(\Omega;\R\right)$. For any given $M >1$ define the stopping times
    \begin{equation} \label{another taumt}
        \tau^{M,T}_n := T \wedge \inf\left\{s \geq 0: \norm{\sy^n}_{X,s}^2 \geq M + \norm{\sy^n}_{X,0}^2 \right\}.
    \end{equation}
Furthermore suppose \begin{equation} \label{supposition 1}
    \lim_{m \rightarrow \infty}\sup_{n \geq m}\mathbbm{E}\left[\norm{\sy^n-\sy^m}^2_{X,\tau
_{m}^{M,t}\wedge \tau_{n}^{M,t}} \right] =0
\end{equation}
and that for any stopping time $\gamma$ and sequence of stopping times $(\delta_j)$ which converge to $0$ $\mathbbm{P}-a.s.$, \begin{equation} \label{supposition 2} \lim_{j \rightarrow \infty}\sup_{n\in\N}\mathbbm{E}\left(\norm{\sy^n}_{X,(\gamma + \delta_j) \wedge \tau^{M,T}_n}^2 - \norm{\sy^n}_{X,\gamma \wedge \tau^{M,T}_n}^2 \right) =0.
\end{equation}
Then there exists a stopping time $\tau^{M,T}_{\infty}$, a process $\sy:\Omega \mapsto X_{\tau^{M,T}_{\infty}}$ whereby $\norm{\sy}_{X,\cdot \wedge \tau^{M,T}_{\infty}}$ is adapted and $\mathbbm{P}-a.s.$ continuous, and a subsequence indexed by $(m_j)$ such that 
\begin{itemize}
    \item $\tau^{M,T}_{\infty} \leq \tau^{M,T}_{m_j}$ $\mathbbm{P}-a.s.$,
    \item $\lim_{j \rightarrow \infty}\norm{\sy - \sy^{m_j}}_{X,\tau^{M,T}_{\infty}} = 0$ $\mathbbm{P}-a.s.$.
\end{itemize}
Moreover for any $R>0$ we can choose $M$ to be such that the stopping time \begin{equation} \nonumber
        \tau^{R,T} := T \wedge \inf\left\{s \geq 0: \norm{\sy}_{X,s\wedge\tau^{M,T}_{\infty}}^2 \geq R \right\}
    \end{equation}
satisfies $\tau^{R,T} \leq \tau^{M,T}_{\infty}$ $\mathbbm{P}-a.s.$. Thus $\tau^{R,T}$ is simply $T \wedge \inf\left\{s \geq 0: \norm{\sy}_{X,s}^2 \geq R \right\}$.

\end{proposition}

\begin{proof}
    See [\cite{goodair2024weak}] Proposition 6.1.
\end{proof}

\begin{remark}
    A consequence of the properties that $\sup_{n}\norm{\sy^n}_{X,0} \in L^\infty\left(\Omega;\R\right)$ and\\ $\lim_{j \rightarrow \infty}\norm{\sy - \sy^{m_j}}_{X,\tau^{M,T}_{\infty}} = 0$ $\mathbbm{P}-a.s.$ is that $\norm{\sy}_{X,0} \in  L^\infty\left(\Omega;\R\right)$. Therefore for\\ $R > \norm{\norm{\sy}_{X,0}}_{L^\infty\left(\Omega;\R\right)}^2$ we have that $\tau^{R,T}$ is $\mathbbm{P}-a.s.$ positive, hence so too is $\tau^{M,T}_{\infty}$ for appropriately chosen $M$.
\end{remark}

\begin{lemma} \label{gronny}
Fix $t>0$ and suppose that $\boldsymbol{\phi},\boldsymbol{\psi}, \boldsymbol{\eta}$ are real-valued, non-negative stochastic processes. Assume, moreover, that there exists constants $c',\hat{c}$ (allowed to depend on $t$) such that for $\mathbbm{P}-a.e.$ $\omega$, \begin{equation} \label{boundingronny} \int_0^t\boldsymbol{\eta}_s(\omega) ds \leq c'\end{equation} and for all stopping times $0 \leq \theta_j < \theta_k \leq t$,
$$\mathbbm{E}\left(\sup_{r \in [\theta_j,\theta_k]}\boldsymbol{\phi}_r\right) + \mathbbm{E}\left(\int_{\theta_j}^{\theta_k}\boldsymbol{\psi}_sds \right)\leq \hat{c}\mathbbm{E}\left(\left(\boldsymbol{\phi}_{\theta_j} + 1 \right) + \int_{\theta_j}^{\theta_k} \boldsymbol{\eta}_s\boldsymbol{\phi}_sds\right) < \infty. $$Then there exists a constant $C$ dependent only on $c',\hat{c},t$ such that $$\mathbbm{E}\left(\sup_{r \in [0,t]}\boldsymbol{\phi}_r\right) + \mathbbm{E}\left(\int_{0}^{t}\boldsymbol{\psi}_sds\right) \leq C\left[\mathbbm{E}(\boldsymbol{\phi}_{0}) + 1\right].$$
\end{lemma}

\begin{proof}
    See [\cite{glatt2009strong}] Lemma 5.3.
\end{proof}

\textbf{Acknowledgements}\\

The author is indebted to Szymon Sobczak for his significant contributions to the paper, particularly in the proof of Proposition 3.2.

\bibliographystyle{newthing}
\bibliography{myBibby}

\end{document}